\documentclass[a4paper,12pt]{amsart}
\usepackage{amsfonts,amsthm,amssymb,amsmath,amscd}
\usepackage[colorlinks=true,linkcolor=blue,citecolor=blue]{hyperref}
\usepackage{amsmath}
\usepackage{tikz-cd}

\usepackage{tabularx}
\usepackage{tcolorbox}
\usepackage[english]{babel}
\usepackage{amssymb,amsmath}
\usepackage{xcolor}

\newtheorem{Theo}{Theorem}[section]
\newtheorem{Prop}[Theo]{Proposition}
\newtheorem{Coro}[Theo]{Corollary}
\newtheorem{Lemm}[Theo]{Lemma}
\newtheorem{Defi}[Theo]{Definition}

\newtheorem{Rema}[Theo]{Remark}

\newcommand{\T}{\mathbb{T}}
\newcommand{\Bcal}{\mathcal{B}}
\newcommand{\D}{\mathbb{D}}
\newcommand{\C}{\mathbb{C}}

\def\N{\mathbb{ N}}
\def\R{\mathbb{ R}}

\begin{document}

\title{Riesz means in Hardy spaces on Dirichlet groups}

\author[Defant]{Andreas Defant}
\address[]{Andreas Defant\newline  Institut f\"{u}r Mathematik,\newline Carl von Ossietzky Universit\"at,\newline
26111 Oldenburg, Germany.
}
\email{defant@mathematik.uni-oldenburg.de}

\author[Schoolmann]{Ingo Schoolmann}
\address[]{Ingo Schoolmann\newline  Institut f\"{u}r Mathematik,\newline Carl von Ossietzky Universit\"at,\newline
26111 Oldenburg, Germany.
}
\email{ingo.schoolmann@uni-oldenburg.de}

\maketitle

\begin{abstract}
Given a frequency $\lambda=(\lambda_n)$, we study when almost all  vertical limits of a  $\mathcal{H}_1$-Dirichlet series $\sum a_n e^{-\lambda_ns}$ are Riesz-summable   almost everywhere on the imaginary axis.
Equivalently, this means to investigate almost everywhere convergence of Fourier series of $H_1$-functions  on so-called
$\lambda$-Dirichlet groups, and as our main technical tool we need to invent a  weak-type $(1, \infty)$  Hardy-Littlewood maximal operator for such groups. Applications are given
to $H_1$-functions on the infinite dimensional torus $\T^\infty$, ordinary Dirichlet series $\sum a_n n^{-s}$, as well as bounded and holomorphic functions  on the open right half plane, which are uniformly almost periodic on every vertical line.
\end{abstract}



\noindent
\renewcommand{\thefootnote}{\fnsymbol{footnote}}
 \footnotetext{}

\section{\bf  Introduction}

 Let  $\lambda$ be
a frequency, i.e. a strictly increasing, unbounded sequence of non-negative real numbers.
Moreover, let $G$ be a compact abelian group, and $\beta: \R \to G$ a continuous homomorphism of groups with dense range such that for each character $e^{-i\lambda_n \cdot}: \R \to \T$ there is a (then unique) character
$h_{\lambda_n}: G \to \T$ with $e^{-i\lambda_n \pmb{\cdot}} =  h_{\lambda_n} \circ\beta$.

\medskip

For
$1 \leq p \leq \infty$ denote by $H^\lambda_p(G)$ the Hardy space of all $f \in L_p(G)$ which have a Fourier
transform $\hat{f}: \hat{G} \to \T$ supported by all characters $h_{\lambda_n}$. It is known that for $1 < p < \infty$ every
$f \in H_p^\lambda(G)$ has an almost everywhere convergent Fourier series representation
$f(\omega) = \sum_{n=1}^{\infty} \widehat{f}(h_{\lambda_n}) h_{\lambda_n}(\omega)$.

\medskip

Inspired by the work \cite{HardyRiesz} of Hardy and Riesz on general Dirichlet series from 1915, we in  this article  study almost everywhere Riesz-summability of the Fourier series of functions $f \in H^\lambda_1(G)$.
The main tool is given by an appropriate weak-type $(1,\infty)$  Hardy-Littlewood
maximal operator.

\medskip
As a particular case we look at the frequency $\lambda = (\log n)$,
the infinite dimensional torus $G= \T^\infty$, and the Kronecker flow $\beta: \R \to G, t \mapsto (p_k^{it})$,
where $p_k$ denotes the $k$th prime. Our results prove that each
$f \in H_{1}(\T^{\infty})$ almost everywhere is the pointwise limit of its logarithmic Riesz means, whereas for  arithmetic Riesz  means (Ces\`{a}ro means) this in general fails.

\medskip

Most of our results  have equivalent formulations in terms of
general  Dirichlet series $\sum a_n e^{-\lambda_n s}$.
More precisely, vertical limits of Dirichlet series $\sum a_{n}e^{-\lambda_{n}s}$ which belong to the Hardy space $\mathcal{H}_{1}(\lambda)$, are summable by their first Riesz means of any order $k>0$ on the
imaginary axis $[\text{Re} =0]$ (and consequently on the right half-plane).

\medskip

Another application shows, that the   Hardy space $H_{\infty}^{\lambda}(G)$ may be identified with the Banach space of all bounded and holomorphic function on $[Re>0]$ which for every $\varepsilon>0$ are  uniformly almost periodic on $[Re>\varepsilon]$, preserving the Fourier and Bohr coefficients.

\medskip

In the following subsections of this introduction  we  substantiate all this and provide our reader with the needed preliminaries. In Section \ref{mainresults} we summarize all our results, and in Section \ref{proofs}
we prove them.

\medskip

\subsection{Hardy spaces on Dirichlet groups} \label{intro}
 Let us briefly recall the general framework of  Hardy spaces $H_p^\lambda(G)$ on so-called
 $\lambda$-Dirichlet groups $(G, \beta)$  from \cite{DefantSchoolmann1}.
\medskip

A pair $(G,\beta)$ of a compact abelian group $G$ and a homomorphism  $\beta \colon (\R,+) \to G$ is said to be a  Dirichlet group, whenever $\beta$ is continuous and has dense range. In this case, the dual map of $\beta$, that is the mapping
$\widehat{\beta} \colon \widehat{G} \hookrightarrow \widehat{\R}, ~~ \gamma \to \gamma \circ \beta$,
is injective, where $\widehat{G}$ denotes the dual group of $G$. So, using the identification $\widehat{(\R,+)}=\R$ (which we do from now on), the dual group of $G$  via $\widehat{\beta}$ can be considered as a subset of $\R$. Moreover, if $x\in \R$ lies in the image of $\widehat{\beta}$, we write $h_{x}:=(\widehat{\beta})^{-1}(x)$ and obtain
$$\widehat{G}=\left\{h_{x} \mid x \in \operatorname{Im} \widehat{\beta}\right\}.$$
In other words, for $x\in \widehat{\beta}(\widehat{G})$ the characters $t \mapsto e^{-ixt}$ on $(\R,+)$ are precisely those, which allow a unique  'extension' $h_{x}\in \widehat{G}$  such that $h_{x}\circ \beta=e^{-ix\cdot}$. Note that we do not force $\beta$ to be injective.
\medskip

Given a frequency $\lambda=(\lambda_{n})$, i.e a strictly increasing non-negative real sequence tending to $\infty$, we call the Dirichlet group $(G,\beta)$ a $\lambda$-Dirichlet group whenever  $\lambda \subset \widehat{\beta}(\widehat{G})$.  Given such a $\lambda$-Dirichlet group  $(G,\beta)$ and $1\le p \le \infty$, we define Hardy space
$$H_{p}^{\lambda}(G):=\left\{ f \in L_{p}(G) \mid \forall~ \gamma \in \widehat{G}:~ \widehat{f}(\gamma)\ne 0 \Rightarrow \gamma=h_{\lambda_{n}} ~\text{for some n} \right\},$$
which being a closed subspace of $L_p(G)$, is a Banach space. Of course, $L_p(G)$ is here formed with respect to the normalized Haar measure $m$ on $G$.
Given two $\lambda$-Dirichlet groups $(G_{1}, \beta_{1})$ and $(G_{2},\beta_{2})$, a crucial fact is that the spaces $H_{p}^{\lambda}(G_{1})$ and $H_{p}^{\lambda}(G_{2})$ are isometrically isomorphic
(see \cite[Corollary 3.21]{DefantSchoolmann1}). More precisely, there is an onto isometry
\begin{equation} \label{crucial}
T \colon H_{p}^{\lambda}(G_{1}) \to H_{p}^{\lambda}(G_{2}), ~~ f \mapsto g,
\end{equation}
which preserves the Fourier coefficients, that is for all $x$ we have
$$\widehat{f}\Big((\widehat{\beta_{1}})^{-1}(x)\Big)=\widehat{g}\Big((\widehat{\beta_{2}})^{-1}(x)\Big).$$
Let us collect a few crucial examples. The Bohr compactification  $\overline{\R}:=\widehat{(\R,d)}$
of $\R$   ($d$ the discrete topology) together with the embedding
\begin{equation*}
\beta_{\overline{\R}}\colon \R \hookrightarrow \overline{\R},\,\, x \mapsto \left[ t \mapsto e^{-ixt}\right]\,,
\end{equation*}
forms a Dirichlet group,
which obviously for any arbitrary frequency $\lambda$ serves as a $\lambda$-Dirichlet group.

\medskip

There are two basic examples which later in many more general situations  will help us to keep orientation.
Consider the frequency  $\lambda=(n)=(0,1,2,\ldots)$. Then $G:=\T$ together with $\beta_{\T}(t):=e^{-it}$ is a $(n)$-Dirichlet group, and $\mathcal{H}_{p}((n))$ equals the classical Hardy space $H_{p}(\T)$.
The second crucial example is   $\lambda=(\log n)$.  In this case, denoting by $\mathfrak{p}=(p_{n})$ the sequence of prime numbers, the infinite dimensional torus  $$\T^{\infty}:=\prod_{n=1}^{\infty} \T$$
(with its natural group structure)
together with the so-called
Kronecker flow
\begin{equation}
\label{oscar}
\beta_{\T^{\infty}}\colon \R \to \T^{\infty}, ~~ t \mapsto \mathfrak{p}^{-it}=(2^{-it},3^{-it}, 5^{-it}, \ldots),
\end{equation}
gives a  $(\log n)$-Dirichlet group. Then
$f \in H_p^{(\log n)}(\T^\infty)$ if and only if $f \in L_p(\T^\infty)$ and $\hat{f}(\alpha) = 0$ for any finite sequence
$\alpha = (\alpha_k)$ of integers with  $\alpha_k < 0$ for some $k$. In other terms,
\[
H_p(\mathbb{T}^\infty) = H_p^{(\log n)}(\T^\infty)
\]
holds isometrically, and $h_{\log n} = z^\alpha$ whenever $n=\mathfrak{p}^\alpha$.

\medskip
There is a useful reformulation of the Dirichlet group $(\T^\infty, \beta_{\T^{\infty}})$. Denote by $\Xi$ the set of
    all characters $\chi: \N \to \T$, i.e. $\chi$ is completely multiplicative in the sense that
    $\chi(nm)=\chi(n)\chi(m)$ for all $n,m$. So every character is uniquely determined by its values on the primes.
    If we  on  $\Xi$ consider
   pointwise multiplication, then
\begin{align*} \label{idy}
\iota\colon \Xi \to \T^{\infty}, ~~ \chi \mapsto \chi(\mathfrak{p}) = (\chi(p_{n}))_{n},
\end{align*}
 is a  group isomorphism which turns  $\Xi$ into a compact abelian group. The Haar measure $d \chi$ on $\Xi$  is the push forward of the normalized Lebesgue measure $dz$ on $\T^\infty$   through $\iota^{-1}$. Hence also $\Xi$ together with
\begin{equation*}
\label{oscar2}
\beta_{\Xi}\colon \R \to \Xi, ~~ t \mapsto [p_k \to \chi(p_k)],
\end{equation*}
forms a $(\log n)$-Dirichlet group.
 \medskip

 Recall from \cite[Lemma 3.11]{DefantSchoolmann1} that, given a Dirichlet group $(G, \beta)$ and $f\in L_{1}(G)$, for almost all $\omega \in G$  there are locally Lebesgue integrable functions $f_{\omega} \colon \R \to \C$ such that $f_{\omega}(t)=f(\omega \beta(t))$ almost everywhere on $\R$.    As we will see later, this way to 'restrict' functions on the group $G$ to $\R$, establishes a sort of  bridge
between  Fourier analysis on a $\lambda$-Dirichlet group $G$ and  Fourier analysis on $\R$.
\medskip

In this context, the classical Poisson kernel  $P_{u}(t)=\frac{1}{\pi}\frac{u}{u^{2}+t^{2}}\colon \R \to \R$, where $u>0$, plays a crucial role. Since $P_u \in L_{1}(\R)$ has norm $1$, its push forward under $\beta$
leads to a  regular Borel probability measure $p_u$ on  $G$. We call $p_u$ the Poisson measure on $G$,
and note that $\widehat{p_{u}}(h_{x})=e^{-u|x|}$ for all $x\in \widehat{\beta}(\widehat{G})$.

\medskip

\subsection{The reflexive case -- functions}
Given a compact and abelian group $G$ with Haar measure $m$ and a class $\mathcal{F}$ of functions in $ L_p(G), 1 \leq p \leq \infty$, one of the fundamental questions
of Fourier analysis certainly is to ask for necessary and sufficient conditions on $\mathcal{F}$ under
which the Fourier series $\sum_{\gamma \in \widehat{G}} \widehat{f}(\gamma) \gamma$ of each $f \in \mathcal{F}$ approximates $f$
in a reasonable way -- e.g.  almost everywhere pointwise or in the $L_p$-norm,
and with respect to a reasonable summation method like ordinary  or Ces\'aro summation.

\medskip

In the following we will carefully distinguish the reflexive case $1 < p < \infty$ from the
non-reflexive cases $p=1$ and $\infty$.

  \begin{Theo} \label{Duy}
  Let  $(G, \beta)$ be a $\lambda$-Dirichlet group and $1<p<\infty$. There is a constant $\text{CH}_p > 0$ such that
    for every $f \in H_{p}^\lambda(G)$ we have
\begin{equation*}
\Big( \int_{G} \sup_{x} \Big| \sum_{\lambda_n < x} \widehat{f}(h_{\lambda_n}) h_{\lambda_n}(\omega)  \Big|^{p} d\omega \Big)^{\frac{1}{p}} \le \text{CH}_p \,\|f\|_{p}.
\end{equation*}
  In particular, $\sum \widehat{f}(h_{\lambda_n}) h_{\lambda_n}$ approximates $f$ almost everywhere pointwise
  and in the $H_{p}$-norm.
   \end{Theo}
  For $\lambda = (n)$ and the Dirichlet group $(\T, \beta_\T)$ this is the celebrated Carleson-Hunt theorem,
  and the constant $\text{CH}_p$ mentioned above in fact equals the one from the maximal inequality in
  the CH-theorem for one variable. Based on this one variable case and a method from \cite{Fefferman}, Hedenmalm-Saksman in
  \cite[Theorem 1.5]{HedenmalmSaksman} extend the CH-theorem  to functions $f \in H_2(\T^\infty)$, which in the preceding theorem is reflected by the  $(\log n)$-Dirichlet group
  $(\T^\infty, \beta_{\T^\infty})$. The general case given above has to be  credited to Duy from \cite{Duy}; for our  reformulation within the  setting of arbitrary $\lambda$-Dirichlet groups we refer to \cite{DefantSchoolmann2}.

  \medskip

The CH-theorem in one variable fails for $p=1$, so clearly the preceding extension fails in this case.
On the other hand, it is well-known that every function $f\in H_{1}(\T)$
 almost everywhere equals the pointwise limit  of its  Ces\`{a}ro means (see e.g. \cite[Theorem 3.4.4, p.207]{Grafakos1}), i.e.
\begin{equation} \label{cesarointro}
f(z) =\lim_{N\to \infty}  \frac{1}{N}  \sum_{k=0}^{N-1} \sum_{n \le k} \widehat{f}(n)z^{n}=\lim_{N\to \infty} \sum_{n=0}^{N-1} \widehat{f}(n) \Big(1-\frac{n}{N}\Big) z^{n}
\end{equation}
for almost all $z \in \T$. Moreover, this is also true  if the limits are  taken with respect to the $H_1$-norm.
So it
  seems natural to consider for a given $f\in H_{1}(\T^{\infty})$ the Ces\`{a}ro means of the partial
  sums $$ \sum_{\alpha: 1\le \mathfrak{p}^{\alpha}\le k} \widehat{f}(\alpha) z^{\alpha}, \,\,k \in \N\,,$$
   and to ask whether  almost everywhere  pointwise and/or in the $H_1$-norm
 \begin{equation} \label{cesarologn}
f(z)=\lim_{N\to \infty}  \frac{1}{N}  \sum_{k=0}^{N-1} \sum_{\mathfrak{p^\alpha} \le k} \widehat{f}(\alpha) z^{\alpha} \,\,?
\end{equation}
We will later see  that this is in general false -- but true, if we change Ces\`{a}ro summation by more adapted   summation methods
invented in \cite{HardyRiesz} by  Hardy and M. Riesz  within the setting of general Dirichlet series.

\medskip

\subsection{Riesz means -- functions} \label{RieszMeans}
The following definitions are inspired by  \cite{HardyRiesz}. Let $\lambda$ be a frequency, $k \ge 0$, and
$\sum c_n$ a series in a Banach space $X$. Then we call the series $\sum c_n$  $(\lambda,k)$-Riesz summable if
the limit
\[
\lim_{x \to \infty} \sum_{\lambda_n < x} \Big(1- \frac{\lambda_n}{x} \Big)^k c_n
\]
exists, and we call the finite sums
\[
R_x^{\lambda, k} \Big( \sum c_n \Big) := \sum_{\lambda_n < x} \Big(1- \frac{\lambda_n}{x} \Big)^k c_n
\]
first $(\lambda,k)$-Riesz means of $\sum c_n $ of length $x > 0$.
\medskip

 Hardy and Riesz in \cite{HardyRiesz}
isolated the following fundamental properties of Riesz summability; the results are (in the order of the proposition) taken from  \cite[Theorem 16, p. 29, Theorem 17, p. 30, and Theorem 21, p. 36]{HardyRiesz}.

\begin{Prop} \label{MarcosgeburtstagX}
Let $\lambda$ be a frequency, $k \ge 0$, and $\sum c_n$ a series in a Banach space $X$.
\begin{itemize}
\item[(1)]
 First theorem of consistency: If $\sum c_n$  is $(\lambda,k)$-Riesz summable, then it is
  $(\lambda, \ell)$-Riesz summable for any $k \leq \ell$, and the associated limits coincide.
  In particular, if $\sum c_n$ is summable (i.e. the series converges),  then for all $k>0$
  \[
  \sum_{n=1}^\infty c_n =
  \lim_{x \to \infty} \sum_{\lambda_n < x} \Big(1- \frac{\lambda_n}{x} \Big)^k c_n.
  \]
  \item[(2)]
Second theorem of consistency:
 If $\sum c_n$  is $(e^\lambda,k)$-Riesz summable, then
  $\sum c_n$  is $(\lambda,k)$-Riesz summable,  and the associated limits coincide.

  \medskip
    \item[(3)]
If $\sum c_{n}$ is $(\lambda, k)$-summable summable with limit $C$, then
\begin{equation*} \label{fromfirsttopartialsum}
\lim_{N\to \infty} \Big(\frac{\lambda_{N+1}-\lambda_{N}}{\lambda_{N+1}}\Big)^{k} \Big(\sum_{n=1}^{N}c_n-C\Big)=0\,.
\end{equation*}
\end{itemize}
\end{Prop}

Note that,
 if $\sum c_n$  in view of Proposition \ref{MarcosgeburtstagX}, (2) is not only
 $(\lambda,k)$-Riesz summable but even $(e^\lambda,k)$-Riesz summable, then
\begin{equation} \label{turmalet}
 \lim_{x \to \infty}
R_{x}^{e^\lambda,k}\big(\sum c_n\big)
=
\lim_{x \to \infty}
R_{e^x}^{e^\lambda,k}\big(\sum c_n\big)
=
\lim_{x \to \infty} \sum_{\lambda_n < x} \Big(1- \frac{e^{\lambda_n}}{e^x} \Big)^k
 c_n \,;
\end{equation}
we refer to  the finite sums
\[
S_x^{\lambda, k} \Big( \sum c_n \Big) :=
\sum_{\lambda_n < x} \Big(1- \frac{e^{\lambda_n}}{e^x} \Big)^k c_n
\]
as the second  $(\lambda,k)$-Riesz means of $\sum c_n $ of length $x > 0$.

\medskip
Take now some $\lambda$-Dirichlet group $(G, \beta)$ and $f\in H_{1}^{\lambda}(G)$. Then we call
the Fourier series of $f$ $(\lambda, k)$-Riesz summable in $\omega \in G$ if it is
$(\lambda, k)$-Riesz summable in $\omega \in G$, in other terms the limit
\[
\lim_{x \to \infty} \sum_{\lambda_n < x} \hat{f}(h_{\lambda_n})\Big(1- \frac{\lambda_n}{x} \Big)^k
 h_{\lambda_n}(\omega)
\]
exists. It is then needless to say what is meant by the phrase 'the Fourier series of $f$ is $(\lambda, k)$-Riesz summable in the $H_p$-norm'. Moreover, the polynomial
\[
R_x^{\lambda,k}(f) := \sum_{\lambda_n < x} \hat{f}(h_{\lambda_n})\Big(1- \frac{\lambda_n}{x} \Big)^k
 h_{\lambda_n}
\]
is the so-called first $(\lambda, k)$-Riesz mean of $f$ of length $x >0$ , and
\[
S_x^{\lambda,k}(f) :=  \sum_{\lambda_n < x} \hat{f}(h_{\lambda_n})\Big(1- \frac{e^{\lambda_n}}{e^x} \Big)^k\,,
\]
the second $(\lambda, k)$-Riesz mean of $f$ of length $x >0$. Observe that if
the Fourier series of $f$ is $(e^\lambda, k)$-Riesz summable in $\omega \in G$, then as in \eqref{turmalet}
\begin{equation} \label{turmalet2}
\lim_{x \to \infty} R_x^{e^\lambda,k}(f)(\omega) = \lim_{x \to \infty} R_{e^x}^{e^\lambda,k}(f)(\omega)=\lim_{x \to \infty} S_x^{\lambda,k}(f)(\omega).
\end{equation}

\medskip

Let us again for a moment  concentrate on the two in a sense extrem frequencies  $\lambda = (n)$ and $\lambda = (\log n)$.

\medskip

As mentioned  $\lambda=(n)$  together with  $(\T, \beta_{\T})$
forms a $\lambda$-Dirichlet group. Then the
 $(\lambda, 1)$-Riesz mean of
  $f \in H_{1}(\T) = \mathcal{H}_{1}^{\lambda}(\T)$ of
   length $x$ equals
\begin{equation} \label{france1}
R_{x}^{\lambda,1}(f)=\sum_{n<x} \widehat{f}(n) \Big(1-\frac{n}{x}\Big) z^{n},
\end{equation}
which for $x=N \in \N$ is nothing else than the Ces\`{a}ro mean of the $N$th partial sum of the Fourier series of $f$ considered in  \eqref{cesarointro}.

\medskip
In this sense Riesz means generalize the Ces\`{a}ro means for  functions  on $\T$ to the much wider setting of functions on Dirichlet groups.

\medskip

 Let us also consider $\lambda = (\log n)$ and the $\lambda$-Dirichlet group
$(\T^\infty, \beta_{\T^\infty})$.
For  $f \in H_{1}(\T^\infty) = H_{1}^{\lambda}(\T^\infty)$ we refer to the first $(\lambda,k)$-Riesz means of $f$, that are
\begin{equation} \label{france2}
R_{x}^{\lambda,k}(f)=\sum_{\log \mathfrak{p}^\alpha<x} \widehat{f}(\alpha)
\Big( 1- \frac{\log \mathfrak{p}^\alpha}{x}\Big)^{k}z^\alpha,~ ~x>0,
\end{equation}
as the logarithmic means of $f$.
Observe also, that in this case
\begin{align} \label{france3}
R_{x}^{e^\lambda,k}(f)(z)=\sum_{\mathfrak{p}^{\alpha}<x} \widehat{f}(\alpha) \Big(1-\frac{\mathfrak{p}^{\alpha}}{x}\Big)^{k} z^{\alpha}\,,
\end{align}
and hence for $N \in \N$ and $k=1$
\begin{align} \label{france3}
R_{N}^{e^\lambda,1}(f)(z)
= \sum_{\mathfrak{p^\alpha} < N} \hat{f}(\alpha) \big(1-\frac{\mathfrak{p}^\alpha}{N}\big)  z^\alpha
=
\frac{1}{N}  \sum_{k=0}^{N-1} \sum_{\mathfrak{p^\alpha} \leq k} \hat{f}(\alpha) z^\alpha.
\end{align}

\begin{Rema} \label{alaphilipX}
Let  $f \in H^\lambda_1(G)$,  $\omega \in G$, and $k \ge 0$.
 Then  for $\lambda = (n)$ the following are  equivalent:
\begin{itemize}
\item[(1)]
The Fourier series of $f$ converges at $\omega$.
\item[(2)]
The Fourier series of $f$ is $(e^\lambda, k)$-Riesz summable at $\omega$.
\end{itemize}
In this case the limits coincide,
and a similar result holds true, whenever we replace convergence in $\omega$ by  convergence with respect to the
$H_p$-norm, $1 \leq p \leq \infty$.
\end{Rema}

\begin{proof}  Part (3) of Proposition~\ref{MarcosgeburtstagX} proves the implication (2) $\Rightarrow$ (1), and the reversed direction follows from Proposition~\ref{MarcosgeburtstagX} (1).
\end{proof}

So for the frequency $\lambda = (n)$, Riesz summability by second means seems not to be particularly interesting.

\medskip

After all this, let  us finally indicate the main challenge of this article:
 For which frequencies $\lambda$ and which $\lambda$-Dirichlet groups $(G, \beta)$ do we for all $f\in H_{1}^{\lambda}(G)$ have
\begin{align} \label{mainQ}
f = \lim_{x \to \infty} R_{x}^{\lambda,k}(f)
\end{align}
 almost everywhere on $G$ and/or in the $H_1$-norm?

\medskip

\subsection{Hardy spaces of Dirichlet series}
Given a frequency $\lambda = (\lambda_n)$,
a $\lambda$-Dirichlet series  is a (formal) sum of the form $D=\sum a_{n}e^{-\lambda_{n}s}$, where $s$ is a complex variable and the sequence $(a_{n})$ form the so-called  Dirichlet coefficients of $D$.
Finite sums of the form $D=\sum_{n=1}^{N} a_{n}e^{-\lambda_{n}s}$ we call Dirichlet polynomials.
  By $\mathcal{D}(\lambda)$ we denote the space of all $\lambda$-Dirichlet series. It is well-known that
  if $D=\sum a_{n}e^{-\lambda_{n}s}$ converges in $s_0 \in \C$, then it also converges for all $s\in \C$
  with $Re s > Re s_0$, and its limit function $f(s) = \sum_{n=1}^{\infty} a_{n}e^{-\lambda_{n}s}$
  defines a holomorphic function on $[Re > \sigma_c(D)]$, where
  \[
  \sigma_{c}(D)=\inf\left \{ \sigma \in \R \mid D \text{ converges on } [Re>\sigma] \right\}
  \]
determines the so-called abscissa of convergence.
\medskip

In \cite{DefantSchoolmann1} we introduce an $\mathcal{H}_{p}$-theory of general Dirichlet series extending Bayart's $\mathcal{H}_{p}$-theory of ordinary Dirichlet series $\sum a_{n}n^{-s}$ (where $\lambda=(\log n)$) from \cite{Bayart} (see also e.g. \cite{Defant}, \cite{Helson}, and \cite{QQ} for more information on the 'ordinary' case). Let us briefly recall the general framework from \cite{DefantSchoolmann1}, which in particular shows, that there are several ways to produce Dirichlet series.
\medskip

Fixing a $\lambda$-Dirichlet group $(G,\beta)$, we define $\mathcal{H}_{p}(\lambda)$, where $1\le p \le \infty$, to be the space of all (formal) $\lambda$-Dirichlet series $D=\sum a_{n}e^{-\lambda_{n}s}$ for which there is $f \in H_{p}^{\lambda}(G)$ such that that  $a_{n}=\widehat{f}(h_{\lambda_{n}})$ for all $n\in \N$. Endowed with the  norm $\|D\|_{p}:=\|f\|_{p}$, we obtain a Banach space.

\medskip

Note that by \eqref{crucial}, the definition of $\mathcal{H}_{p}(\lambda)$ is independent of the chosen $\lambda$-Dirichlet group, and we by definition obtain the onto isometry
\begin{equation} \label{Bohrmap}
\mathcal{B} \colon H_{p}^{\lambda}(G) \to \mathcal{H}_{p}(\lambda), ~~ f \mapsto \sum \widehat{f}(h_{\lambda_{n}}) e^{-\lambda_{n}s}\,;
\end{equation}
for  historical reasons we  call this mapping Bohr transform.
From \cite[Theorem 3.26]{DefantSchoolmann1} recall the following  internal description of $\mathcal{H}_{p}(\lambda)$: Since $D=\sum_{n=1}^{N}a_{n}e^{-\lambda_{n}s}$, considered as a function on the imaginary line, defines an almost periodic function, the limit
\begin{equation}\label{internal}
\|D\|_{p}:=\lim_{T\to \infty} \Big(\frac{1}{2T} \int_{-T}^{T} \Big| \sum_{n=1}^{N}a_{n}e^{-\lambda_{n}it}\Big|^{p} dt \Big)^{\frac{1}{p}}
\end{equation}
exists and defines a norm on the space $Pol(\lambda)$ of all $\lambda$-Dirichlet polynomials. Then $\mathcal{H}_{p}(\lambda)$ is the completion of $\big(Pol(\lambda),\|\cdot\|_{p}\big)$.

\medskip

\subsection{Transference}
We  want  to understand, how  (pointwise) summability properties of the Fourier series $\sum \widehat{f}(h_{\lambda_{n}}) h_{\lambda_{n}}$
  of functions $f \in H_{1}^{\lambda}(G)$ transfer to  summability properties of their associated Dirichlet series $D:=\mathcal{B}(f)$, and vice versa. Slightly  more precise, but still vague,
we try to figure out how  summation of these Fourier series by first or second Riesz means influences the
 convergence properties
of  so-called vertical limits of $D$.

\medskip

Given a $\lambda$-Dirichlet series $D = \sum a_n e^{-\lambda_n s}$ and $z \in \mathbb{C}$, we call
the Dirichlet series $$D_{z}:=\sum a_n e^{-\lambda_n z} e^{-\lambda_n s} $$ the   translation of $D$ about $z$, and we distinguish between horizontal translations $D_u, u \in \R$,  and vertical translations
$D_{i\tau}, \tau \in \R$.

\medskip

 If $(G, \beta)$ is a $\lambda$-Dirichlet group and $D \in \mathcal{H}_{p}(\lambda)$ is associated to
 $f\in H_p^\lambda(G)$, then for each $u>0$
 the  horizontal  translation $D_u$ corresponds to the convolution of $f$ with the Poisson measure $p_{u}$, i.e.
 $\Bcal(f*p_{u})=D_{u}$ (compare coefficients). In particular,  we have that $D_{u}\in \mathcal{H}_{p}(\lambda)$.

\medskip

Moreover, each Dirichlet series of the form $$D^{\omega}=\sum a_{n} h_{\lambda_{n}}(\omega)e^{-\lambda_n s}\,,\,\, \omega \in G$$
is said to be a vertical limit of $D$. Examples are vertical translations
$D_{i\tau}$ with $\tau \in \mathbb{R}$,
and the terminology is explained by the fact that each vertical limit  may be approximated by  vertical translates. More precisely, given $D =  \sum a_n e^{-\lambda_n s}$ which converges absolutely on the right half-plane, for every  $\omega \in G$ there is a sequence $(\tau_{k})_{k} \subset \R$ such that $D_{i\tau_{k}}$ converges to $D^{\omega}$ uniformly on $[Re>\varepsilon]$ for all $\varepsilon>0$.
 Assume conversely that for  $(\tau_{k})_{k} \subset \R$ the  vertical translations $D_{i\tau_k}$ converge
 uniformly on $[Re>\varepsilon]$ for every  $\varepsilon>0$
 to a holomorphic function $f$ on $[Re>0]$. Then there is $\omega \in G$ such that
    $f(s)= \sum_{n=1}^\infty a_n \omega(h_{\lambda_n}) e^{-\lambda_n s}$
    for all $s \in [Re>0]$\,. For all this see \cite[Proposition 4.6]{DefantSchoolmann1}.

    \medskip

The following lemma (to be proved in Section \ref{mainproof2}) is our 'bridge' comparing
almost everywhere Riesz-summability  of the Fourier series of a function $f\in H_1^\lambda(G)$
with the convergence of almost all  vertical limits $D^\omega$ of its associated Dirichlet series
$D= \mathfrak{B}(f)$ almost everywhere on the  imaginary axis.

\begin{Lemm} \label{hedesaks} Let $(G,\beta)$ be a $\lambda$-Dirichlet group, and $f_n, f \in H_1^{\lambda}(G)$. Then the following are equivalent:
\begin{itemize}
\item[(1)]
$\lim_{n\to \infty} f_n(\omega) = f(\omega)$ \,\,\, \text{for almost all $\omega \in G$}
\vspace{1mm}
\item[(2)]
$\lim_{n\to \infty} (f_n)_\omega(t)= f_\omega(t)$ \,\,\,
\text{for almost all $\omega \in G$ and for  almost all $t\in \R$.}
\end{itemize}
In particular, if all $f_n$ are polynomials and $D_n\in \mathcal{H}_1(\lambda)$
 are the Dirichlet polynomials  associated  to $f_n$ under the Bohr transform, then $(1)$  and $(2)$ are equivalent to each of the
 following two further statements:
\begin{itemize}
\item[(3)]
$\lim_{n\to \infty} D^\omega_n(0)= f(\omega)$ \,\,\,
\text{for almost all $\omega \in G$}
\vspace{1mm}
\item[(4)]
$\lim_{n\to \infty} D^\omega_n(it)= f_\omega(t)$ \,\,\,
\text{for almost all $\omega \in G$ and for  almost all $t\in \R$}\,.
\end{itemize}
\end{Lemm}

Note that the question we formulated in \eqref{mainQ} then reads: For which frequencies $\lambda$
and for which  $\lambda$-Dirichlet groups $(G,\beta)$ is it true that for
every $D \in \mathcal{H}_1(\lambda)$ we for
almost all $\omega \in G$ have
\begin{align} \label{D-mainQ}
D^\omega = \lim_{x \to \infty} R_{x}^{\lambda,k}(D^\omega)
\end{align}
almost everywhere on the imaginary axis?
\medskip

\subsection{The reflexive case -- Dirichlet series}
The following result is  an immediate  consequence of Theorem \ref{Duy} and Lemma \ref{hedesaks}.

  \begin{Theo} \label{Dir-Duy}
  Let $\lambda$ be a frequency, $(G, \beta)$ a $\lambda$-Dirichlet group, and $1 < p < \infty$.
  Then for every   $D = \sum a_n e^{-\lambda_n s} \in \mathcal{H}_p(\lambda)$
\begin{equation*}
\Big( \int_{G} \sup_{x} \Big| \sum_{n=1}^x a_n  h_{\lambda_n}(\omega)\Big|^{p} d\omega \Big)^{\frac{1}{p}} \le \text{CH}_p \,\|D\|_{p}.
\end{equation*}
In particular,  almost all vertical limits $D^\omega$
converge almost everywhere on the imaginary axis, and consequently also on the right half-plane.
  \end{Theo}

\noindent For $p=2$   Helson in \cite[\S2]{Helson3} proves convergence on $[Re > 0]$
 under  Bohr's condition for $\lambda$, i.e.
 \begin{equation*}
 \exists ~l>0 ~\forall ~\delta >0 ~\exists~ C>0~ \forall ~n\in \N \colon ~\lambda_{n+1}-\lambda_{n}\ge C e^{-(l+\delta)\lambda_{n}}
 \end{equation*}
   (see also \cite[Theorem 9, p. 29]{Helson}),
and in the ordinary case and $1\leq p < \infty$ this is done by  Bayart \cite{Bayart}.
Still in the ordinary case, convergence  on the imaginary axis $[Re = 0]$ for  $p=2$  has to be credited to Hedenmalm-Saksman \cite{HedenmalmSaksman}, and for the full scale $1<p<\infty$ this is observed in \cite{DefantSchoolmann2}.

\medskip

But for $p=1$ the first two assertions of the  preceding result are false. Otherwise  by Lemma \ref{hedesaks} we would see that all Fourier series of functions $f \in H_1(\T^\infty)$
 converge pointwise almost everywhere which we know is false (even in the one variable case). The third statement we only know under  the additional Landau condition $(LC)$ for $\lambda$ (see again \cite{DefantSchoolmann2}), i.e.
\begin{equation} \label{LC}
\forall~ \delta>0~ \exists~ C>0~ \forall~ n \in \N \colon ~ \lambda_{n+1}-\lambda_{n}\ge Ce^{-e^{\delta \lambda_{n}}};
\end{equation}
this in fact is a condition weaker  than $(BC)$.

\begin{Theo} \label{Landau}
Let  $\lambda$ be a frequency
    with $(LC)$
    and
    $(G, \beta)$  a $\lambda$-Dirichlet group. Then
for $D = \sum a_n  e^{-\lambda_n s} \in \mathcal{H}_1(\lambda)$
almost all vertical limits $D^\omega$ of $D$ converge on $[\text{Re}>0]$. More precisely, if $f \in H_1^\lambda (G)$ is associated to $D$, then there is a null set $N\subset G$ such that  for $\omega \notin N$, all $u>0$ and almost all $t \in \R$
\begin{equation*} \label{formula}
D^{\omega}(u+it)=\sum_{n=1}^{\infty} a_n h_{\lambda_{n}}(\omega)  e^{-\lambda_n (u+it)} =(f_{\omega}*P_{u})(t) .
\end{equation*}

\end{Theo}

See also Section \ref{secondRieszmeans}, where we show that in the ordinary case Theorem \ref{Dir-Duy} is false for $p=1$, even if  we there replace summation of the series by the weaker  Ces\`{a}ro  summation.
 Alternative  more adapted  summation methods have to be taken into account which we describe now
 (recall also the discussion from \eqref{cesarologn}).
%
%
%

\subsection{Riesz means -- Dirichlet series}
Here we repeat some fundamental definitions and results on Riesz-summability of general Dirichlet series from \cite{HardyRiesz}.
Fix some frequency  $\lambda$, some $k\ge 0$, and a $\lambda$-Dirichlet series $D=\sum a_{n}e^{-\lambda_{n}s}$. The first $(\lambda,k)$-Riesz mean of $D$ of length $x>0$ is given by the Dirichlet polynomial
$$R_{x}^{\lambda,k}(D)(s):=\sum_{\lambda_{n}<x} a_{n} \Big(1-\frac
{\lambda_{n}}{x}\Big)^{k}e^{-\lambda_{n}s}.$$
We say that a Dirichlet series $D=\sum a_{n}e^{-\lambda_{n}s}$ is $(\lambda,k)$-Riesz summable at $s_{0} \in \C$, if the limit
$$\lim_{x\to \infty} R_{x}^{\lambda,k}(D)(s_{0})=\lim_{x\to \infty}\sum_{\lambda_{n}<x} a_{n} \Big(1-\frac{\lambda_{n}}{x}\Big)^{k}e^{-\lambda_{n}s_{0}}$$
exists, and $D=\sum a_{n}e^{-\lambda_{n}s}$ is  $(\lambda,k)$-Riesz summable in $\mathcal{H}_p(\lambda)$
whenever  this limit exists in $\mathcal{H}_p(\lambda)$. As in \eqref{turmalet}, if $D$ is (even) $(e^\lambda,k)$-Riesz summable, then  we have
\begin{equation} \label{turmalet3}
\lim_{x \to \infty} R_x^{e^\lambda,k}(D)(s) = \lim_{x \to \infty} R_{e^x}^{e^\lambda,k}(D)(s)=\lim_{x \to \infty} S_x^{\lambda,k}(D)(s)\,,
\end{equation}
where
\[
S_x^{\lambda, k} (D)(s) :=
\sum_{\lambda_n < x} a_n \Big(1- \frac{e^{\lambda_n}}{e^x} \Big)^k e^{\lambda_n s}
\]
is what we call the  second $(\lambda,k)$-Riesz mean of $D$ of length $x>0$.
\medskip

Let us again comment on the ordinary case $\lambda=(\log n)$. Then the first $(\lambda,k)$-Riesz mean  of $D=\sum a_{n}n^{-s}$ of length $x$ is given  by
$$R_{x}^{\lambda,k}(D)(s)=\sum_{\log n<x} a_{n}\Big(1-\frac{\log n}{x}\Big)^kn^{-s};$$
Hardy and Riesz in \cite{HardyRiesz} call it the logarithmic mean of $D$
of length $x$.
As in \eqref{france3}, we for $\lambda =(\log n)$
and $x = N \in \N$  obtain   Ces\`{a}ro means,
\begin{align} \label{france4}
R_{N}^{e^\lambda,k}(D)(s)
= \sum_{n < N} a_n \big(1-\frac{n}{N}\big) a_n n^{-s}
=
\frac{1}{N}  \sum_{k=0}^{N-1} \sum_{n \leq k} a_n n^{-s}\,.
\end{align}

From Remark \ref{alaphilipX} we may deduce the following equivalence for the case $\lambda = (n)$.

\begin{Rema} \label{sagan2}
Let  $D= \sum a_n e^{-\lambda_ns}$,  $ s_0 \in \C$, and $k \ge 0$.
 Then  for $\lambda = (n)$ the following are  equivalent:
\begin{itemize}
\item[(1)]
$D$ converges at $s_0$
\item[(2)]
$D$  is $(e^\lambda,k)$-summable at $s_0$
\end{itemize}
Moreover, the limits coincide,
and the analog  result holds true, whenever we replace convergence in $s_0$ by  convergence with respect to the
$\mathcal{H}_p$-norm, $1 \leq p \leq \infty$.
\end{Rema}

All results collected in Proposition \ref{MarcosgeburtstagX} may be  translated to
$\lambda$-Dirichlet series. We do this in terms of Riesz-abscissas of convergence.
Define for  $D=\sum a_{n}e^{-\lambda_{n}s}$
\[
\sigma_c^{\lambda,k}(D) = \inf \Big\{  \sigma \in \R \colon
 \text{$D$ is $(\lambda,k)$-Riesz summable at $\sigma$} \Big\}\,.
\]
Then it is proved in \cite[Theorem 24 and 25, p. 43]{HardyRiesz} that $D$ converges on the half-plane $[Re > \sigma_c^{\lambda,k}(D)]$
whereas it  diverges on $[Re < \sigma_c^{\lambda,k}(D)]$. Moreover, the limit function is holomorphic (see \cite[Theorem 27, p. 44]{HardyRiesz}). Obviously, we have  that
\[
\sigma_c^{\lambda,0}(D)= \sigma_c(D)\,.
\]
From Hardy and Riesz \cite[Theorem 16, p. 29 and Theorem 30, p. 45]{HardyRiesz} (see also again Proposition \ref{MarcosgeburtstagX}) we know the following.

\begin{Prop} \label{MarcosgeburtstagY}
Let $\lambda$ be a frequency and $0 \leq k < \ell$. Then for every $\lambda$-Dirichlet series $D$ we have
\begin{itemize}
\item[(1)]
 $\sigma_c^{\lambda,k}(D) \leq \sigma_c^{\lambda,\ell}(D)$
 \medskip
\item[(2)]
$\sigma_c^{e^\lambda,k}(D) = \sigma_c^{\lambda,k}(D)$
\end{itemize}
\end{Prop}
Finally, we note that for $f \in H_1^\lambda(G)$, where $(G, \beta)$ is some $\lambda$-Dirichlet group
and $D \in \mathcal{H}_1(\lambda)$ with $D:=\Bcal(f)$, we for all $k,x>0$ have
$$\mathcal{B}(R_{x}^{\lambda,k}(f))=R_{x}^{\lambda,k}(D).$$

\medskip

\section{\bf Results} \label{mainresults}
We start  summarizing
our results according to the following table of content, and then later on all proofs will be given in  the final Section \ref{proofs}.

\medskip

\begin{itemize}
\item
Pointwise approximation by first Riesz means (Section \ref{appfirst})
\item
Failure of pointwise approximation by second Riesz means (Section \ref{secondRieszmeans})
\item
A Hardy-Littlewood maximal operator (Section \ref{firstRieszmeans})
\item
Norm approximation by  Riesz means  (Section \ref{approxnormchapter})
\item
Uniformly almost periodic functions (Section \ref{aap})
\end{itemize}

\medskip

\subsection{Pointwise approximation   by first Riesz means} \label{appfirst}
Generally speaking, in order to obtain almost everywhere convergence of the Riesz means of some
$f \in H^\lambda_1(G)$, it is sufficient  to prove  an adequate   maximal inequality.

\medskip

For that purpose recall that for some measure space $(\Omega, \mu)$ the weak $L_{1}$-space
$L_{1, \infty}(\mu)$ is given by all measurable functions $f\colon \Omega\to \C$ for which  there is a constant $C>0$ such that for all $\alpha>0$ we have
\begin{equation} \label{Lorentz}
m\big(\left\{ \omega \in G \mid |f(\omega)|>\alpha \right\} \big)\le \frac{C}{\alpha}.
\end{equation}
Together with $\|f\|_{1,\infty}:= \inf C$ the space $L_{1,\infty}(\mu)$ becomes a quasi Banach space (see e.g. \cite[\S 1.1.1 and \S 1.4]{Grafakos1}), where the triangle inequality holds with constant $2$.

\medskip

The following maximal inequality   forms  the core of this article.
\begin{Theo}\label{maximalineqtypeI} Let $\lambda$ be a frequency, $k>0$ and $(G,\beta)$ a $\lambda$-Dirichlet group. Then
$$R_{\max}^{\lambda,k}(f)(\omega):=\sup_{x>0} \big|R_{x}^{\lambda,k}(f)(\omega)\big|=\sup_{x>0} \Big| \sum_{\lambda_{n}<x} \widehat{f}(h_{\lambda_{n}}) \Big(1-\frac{\lambda_{n}}{x}\Big)^{k} h_{\lambda_n}(\omega)\Big|$$
defines a bounded sublinear operator from $H_{1}^{\lambda}(G)$ to $L_{1,\infty}(G)$, and from $H_{p}^{\lambda}(G)$ to $L_{p}(G)$, whenever $1<p\le \infty$.
\end{Theo}
By a standard argument (to be formalized in Lemma \ref{egoroff2}), we deduce from Theorem \ref{maximalineqtypeI}
for each  $f \in H_{1}^{\lambda}(G)$ almost everywhere summability by first Riesz means of the Fourier series
$\sum\widehat{f}(h_{\lambda_{n}}) h_{\lambda_{n}}$,
 but also of the translated Fourier series
$\sum \widehat{f}(h_{\lambda_{n}})e^{-u\lambda_n}  h_{\lambda_{n}}$.
\begin{Coro} \label{almosteverywheremaintheo}
Let $f \in H_{1}^{\lambda}(G)$ and $k>0$. Then for almost  all $\omega \in G$ we have
\begin{equation} \label{conv}
\lim_{x\to \infty} R_{x}^{\lambda,k}(f)(\omega) =f(\omega).
\end{equation}
Moreover, there is a null set $N \subset G$  such that for all $u > 0$ and all $\omega \notin N $
\begin{equation} \label{conv2}
\lim_{x\to \infty} R_{x}^{\lambda,k}(f*p_{u})(\omega)=(f*p_{u})(\omega).
\end{equation}
\end{Coro}

Note that for some fixed $u>0$ the result from \eqref{conv2} is immediate from \eqref{conv}
since $ f*p_{u} \in H_{1}^{\lambda}(G)$,
but the point here is that the null set $N$ in fact can be chosen to be independent of $u$.

\medskip

In particular, (\ref{conv}) reproves (\ref{cesarointro}) with the choice $\lambda=(n)$, $k=1$ and the Dirichlet group $(\T, \beta_{\T})$. Since first Riesz means of order $k=0$ are precisely partial sums, Corollary \ref{almosteverywheremaintheo} may fail for $k=0$ (like it does for $\lambda=(n)$).
Let us also revisit the ordinary case $\lambda=(\log n)$.  Then Theorem \ref{maximalineqtypeI} gives the following substitute for the failure of \eqref{cesarologn}.
\begin{Coro} \label{pino}
Let $f \in H_{1}(\T^{\infty})$ and $k>0$. Then  for almost all $z\in \T^{\infty}$
\begin{equation} \label{logarithmicmeans}
\lim_{x\to \infty} \sum_{\log(\mathfrak{p}^{\alpha})<x} \widehat{f}(\alpha) \Big(1-\frac{\log(\mathfrak{p}^{\alpha})}{x}\Big)^{k}z^{\alpha}=f(z),
\end{equation}
that is, $f$ almost everywhere is the pointwise limit of its logarithmic means.
\end{Coro}

 Let us explain in which sense \eqref{conv} is the limit case of \eqref{conv2}.

\begin{Prop}\label{Fatouthm}
Let $\lambda$ be arbitrary. Then the operator
$$T(f)(\omega):=\sup_{u>0} |(f*p_{u})(\omega)|$$ defines a bounded sublinear operator from $L_{1}(G)$ to $L_{1,\infty}(G)$, and from $L_{p}(G)$ to $L_{p}(G)$, provided $1 < p \le \infty$. In particular,  for almost all $\omega \in G$
$$\lim_{u\to 0} (f*p_{u})(\omega)=f(\omega).$$
\end{Prop}
Obviously,  Proposition \ref{Fatouthm} implies that for almost all $\omega \in G$
\begin{align*}
\lim_{u\to 0} \lim_{x\to \infty} R_{x}^{\lambda,k}(f*p_{u})(\omega)
&
=\lim_{u\to 0} (f*p_{u})(\omega)
\\
&
=f(\omega)=\lim_{x\to \infty} \lim_{u\to 0} R_{x}^{\lambda,k}(f*p_{u})(\omega),
\end{align*}
explaining, why \eqref{conv} is the limit case of \eqref{conv2}.
\medskip

For the  $(n)$-Dirichlet group $(\T, \beta_{\T})$ all this is linked with Fatou's famous theorem on radial limits. Observe that, if $f\in H_{1}(\T)$ and $0<r<1$, then with the choice $u:=-\log(r)$ we have
$$f*p_{u}(z)=\sum_{n=0}^{\infty}\widehat{f}(n) r^{n} z^{n},$$
and Proposition \ref{Fatouthm} implies, that for almost all $z\in \T$
\begin{equation} \label{powerfatou}
\lim_{r\to 1}  \sum_{n=0}^{\infty} \widehat{f}(n)r^{n}z^{n} =f(z).
\end{equation}
 In the terminology  of \cite[Theorem 11.23]{Rudin} this says that
 for  every $f\in H_{p}(\T)$ the so-called Poisson integral
 $P[f](w):=\sum_{n=0}^{\infty} \widehat{f}(n)w^{n}, w \in \D$  has radial limits almost everywhere on $\D$ . This is a crucial step of the proof of Fatou's theorem, which states, that if $ f\in H_{p}(\D)$, $1\le p \le \infty$, then for almost all $z\in \T$ the radial limits
$g^{*}(z):=\lim_{r\to 1} g(rz)$ exist and define a function from $H_{p}(\T)$. So in this sense,  Proposition \ref{Fatouthm} extends (\ref{powerfatou}) from $H_{p}^{(n)}(\T)$ to $H_{p}^{\lambda}(G)$.
\medskip

Let us again revisit Theorem \ref{maximalineqtypeI}. If $p=1$, then the function $R_{\max}^{\lambda,k}(f)$
for $f \in H^\lambda_1(G)$ may fail to be in $L_{1}(G)$, like it does for $\lambda=(n)$ and $(\T,\beta_{\T})$. On the other hand, philosophically speaking, for $f \in H^\lambda_1(G)$ a horizontal translation of $f$ by $u>0$, that is
\[
 f_u = f \ast p_u \sim \sum \widehat{f}(h_{\lambda_{n}})e^{-u\lambda_{n}} h_{\lambda_{n}},
 \]
 improves $f$ considerably. Then, as shown in the following result, $R_{\max}^{\lambda,k}(f_u)$ for any $u,k>0$ indeed belongs
 to $L_1(G)$. Of course we we already know   from Theorem~\ref{maximalineqtypeI} that $R_{\max}^{\lambda,k}(f_u)\in L_{p}(G)$ whenever $f\in H_{p}^{\lambda}(G)$ and $1 < p \le \infty$.

\begin{Theo}  \label{maxitranslatedtypeI} Let $u,k>0$. Then there is a constant $C=C(u,k)$ such that for all $\lambda$, $1\le p <\infty$ and $f \in H_{p}^{\lambda}(G)$ we have
$$\Big(\int_{G} \sup_{x>0} \Big| \sum_{\lambda_{n}<x} \widehat{f}(h_{\lambda_{n}})e^{-u\lambda_{n}} \Big(1-\frac{\lambda_{n}}{x}\Big)^{k} h_{\lambda_n}(\omega) \Big|^{p} d\omega\Big)^{\frac{1}{p}} \le C\|f\|_{p}.$$
\end{Theo}

 Compared with Theorem \ref{maximalineqtypeI} the relevant part of this result is the case $ p=1$. But also the case  of arbitrary $p$'s seems interesting since in the above inequality
the constant $C=C(u,k)$  does not depend on $p$.

\medskip

Finally, we discuss Riesz summation of Dirichlet series. We use the 'bridge' from Lemma \ref{hedesaks} to rephrase  Corollary \ref{almosteverywheremaintheo} in terms of the summability of almost all vertical limits of $\mathcal{H}_1$-Dirichlet series by first Riesz means almost everywhere  on the imaginary axis.

\begin{Coro} \label{Dirichletseries} Let $\lambda$ be a frequency, $k >0$,  $f\in H_{1}^{\lambda}(G)$,
and $D$ its associated Dirichlet series in $\mathcal{H}_1(\lambda)$.  Then
there is a null set $N \subset G$ such that for all $\omega \notin N$
\begin{equation} \label{coro1.3zero}
\lim_{x\to\infty} R_{x}^{\lambda,k}(D^{\omega})(0) = f(\omega)\,,
\end{equation}
\begin{equation} \label{coro1.3zeroo}
\lim_{x\to\infty} R_{x}^{\lambda,k}(D^{\omega})(it)=f_{\omega}(t)\,\,\, \text{ for almost all $t \in \R$ },
\end{equation}
and
\begin{equation} \label{coro1.3pos}
\lim_{x\to \infty} R_{x}^{\lambda,k}(D^{\omega})(u+it)=f_{\omega}*P_{u}(t)\,\,\,
\text{ for all $u>0$ and  almost all $t \in \R$. }
\end{equation}
\end{Coro}
Note that by Theorem \ref{Landau} the case $k=0$ in \eqref{coro1.3pos} is known, if $\lambda$ satisfies $(LC)$.

\medskip

In view of \eqref{internal} the following maximal inequality may be considered as an
internal variant of Theorem \ref{maxitranslatedtypeI}.

\begin{Theo} \label{maxitransinternaltype1}
Let $u,k>0$. Then there is a constant $C=C(u,k)$ such that for all $\lambda$, $1\le p <\infty$, and $D=\sum a_{n}e^{-\lambda_{n}s}\in \mathcal{H}_{p}(\lambda)$ we for almost all $\omega \in G$ have
\begin{equation*}
\lim_{T\to \infty} \Big( \frac{1}{2T} \int_{-T}^{T} \sup_{x>0} \Big| \sum_{\lambda_{n}<x} a_{n} h_{\lambda_n}(\omega) \Big(1-\frac{\lambda_{n}}{x}\Big)^{k}e^{-\lambda_{n}(u+it)} \Big|^{p} dt \Big)^{\frac{1}{p}} \le C\|D\|_{p}.
\end{equation*}
\end{Theo}

All proofs of this section are given in the Sections \ref{reductionsubsection}, \ref{mainproof},  \ref{mainproof2},
and  \ref{translatedsubsection}.

\medskip

\subsection{Failure of pointwise approximation by second Riesz means}  \label{secondRieszmeans}
Let $\lambda$ be a frequency, $k>0$, and $(G,\beta)$ a $\lambda$-Dirichlet group.
From Corollary \ref{almosteverywheremaintheo} we know that the Fourier series of each $f \in H_{1}^\lambda(G)$
is $(\lambda, k)$-Riesz summable with limit $f$, i.e. $f$ is almost everywhere  approximable by its first Riesz means $R^{\lambda,k}_x(f)$.  Let us discuss whether in view of Proposition \ref{MarcosgeburtstagX}, (2)  the Fourier series of each $f \in H_{1}^\lambda(G)$ is even $(e^\lambda,k)$-Riesz summable with limit $f$,  i.e. $f$  is almost
everywhere approximable by its second Riesz means $S^{\lambda,k}_x(f)$ (see \eqref{turmalet2}).

\medskip

For $\lambda = (n)$  and the $\lambda$-Dirichlet group $(\T, \beta_{\T} )$
we know that the Fourier series of  each $f \in H_1^{\lambda} (\T)=H_1^{\lambda} (\T)$ is Ces\'aro summable almost everywhere with limit $f$, i.e.
it is $(\lambda, 1)$-Riesz summable with limit $f$ (as mentioned this is a particular case of
Corollary \ref{almosteverywheremaintheo}). But it is definitely not $(e^\lambda,1)$-Riesz summable since otherwise
by Proposition \ref{alaphilipX} each $f\in H_1(\T)$ would have an almost everywhere convergent Fourier series.

\medskip

On the other hand look at  $\lambda = (\log n)$  and the $\lambda$-Dirichlet group $(\T^\infty, \beta_{\T^\infty})$.
We have already mentioned Corollary \ref{pino} which shows that the Fourier series of every
 $f \in H_1(\T^\infty) = H_1^\lambda(\T^\infty)$ is $(\lambda,1)$-Riesz summable with limit $f$. Assume that the Fourier series of each such
 $f$ may even  be almost everywhere
 $(e^\lambda, 1)$-summable, i.e. it
 is almost everywhere on $\T^\infty$ approximable  by its second Riesz means:

\[
f = \lim_{x\to \infty}S_{x}^{\lambda,k}(f)
= \lim_{x\to \infty} \sum_{\mathfrak{p}^\alpha < e^x} \hat{f}(\alpha) \big(1-\frac{\mathfrak{p}^\alpha}{e^x}\big) z^\alpha
=  \lim_{x\to \infty} \sum_{\mathfrak{p}^\alpha < x} \hat{f}(\alpha) \big(1-\frac{\mathfrak{p}^\alpha}{x}\big) z^\alpha
\,.
\]

Then for every $f \in H_1(\T) \subset H_1(\T^\infty) $ almost everywhere on $\T$
\[
f
=  \lim_{x\to \infty} \sum_{2^j < x} \hat{f}(j) \big(1-\frac{2^j}{x}\big) z^j
\,.
\]
In other words the Fourier series of every $f \in H_1(\T)$ is almost everywhere $( (2^j), 1)$-Riesz summable
with limit $f$. By Proposition \ref{MarcosgeburtstagX}, this implies that each such Fourier series
is almost everywhere summable on $\T$, a contradiction.

\medskip

We collect the preceding information in the following

\begin{Rema} \label{bardet}
Given a frequency $\lambda$ and a $\lambda$-Dirichlet group $(G, \beta)$, it is in general not true,
that each $f \in H_1^\lambda(G)$ has an almost everywhere $(e^\lambda,1)$-summable Fourier series.
Counterexamples are $\lambda = (n)$ and $(\T, \beta_\T)$, as well as $\lambda = (\log n)$  and $(\T^\infty, \beta_{\T^\infty})$.
\end{Rema}

Corollary  \ref{Dirichletseries} shows that, given a $\lambda$-Dirichlet group $(G, \beta)$ and $k >0$, almost all vertical limits $D^\omega$ of a Dirichlet series $D \in \mathcal{H}_1(\lambda)$ are $(\lambda,k)$-Riesz summable
almost everywhere on the vertical line. Are they even $(e^\lambda,k)$-Riesz summable? Remark \ref{bardet} and
Lemma \ref{hedesaks} show that the answer in general in no.

\medskip

But something can be saved. By Corollary  \ref{Dirichletseries} and Proposition \ref{MarcosgeburtstagY}  almost all vertical limits of Dirichlet series $D\in \mathcal{H}_{1}(\lambda)$ are $(\lambda,k)$-Riesz summable on $[Re >0]$, and so
$(e^\lambda,k)$-Riesz summable  on $[Re>0]$. As a consequence of Theorem \ref{maxitranslatedtypeI} we are going to deduce the following quantitative version of this fact.

\begin{Theo} \label{maxitranslatedtypeII}
Let $u,k>0$. Then there is a constant $C=C(u,k)$ such that for all $\lambda$, $1\le p <\infty$ and $f \in H_{p}^{\lambda}(G)$ we have
$$\left(\int_{G}\sup_{x>0} \Big| \sum_{\lambda_{n}<x}  \widehat{f}(h_{\lambda_{n}}) e^{-u\lambda_{n}} \Big(1-e^{\lambda_{n}-x}\Big)^{k} h_{\lambda_n}(\omega) \Big|^{p} d\omega\right)^{\frac{1}{p}}
\le C \|f\|_{p}.$$
\end{Theo}

One could be tempted to believe that  this  maximal inequality is  an immediate consequence of  Theorem \ref{maxitranslatedtypeI} (applying it to $e^\lambda$ instead of $\lambda$). But this is not true since we here
consider functions in
$H_{p}^{\lambda}(G)$ and not  $H_{p}^{e^\lambda}(G)$.
The proof  is given in Section \ref{translatedsubsection}.

\medskip

\subsection{A Hardy-Littlewood maximal operator} \label{firstRieszmeans}
One of the central tools needed for the proof of Theorem \ref{maximalineqtypeI} is given a Hardy-Littlewood maximal operator adapted to Fourier analysis  on Dirichlet groups. If $f \in L_{1}(G)$, where $(G, \beta)$ is any Dirichlet group, then we define
\begin{equation} \label{maximaloperatordef}
\overline{M}(f)(\omega):=\sup_{I\subset \R} \frac{1}{|I|}\int_{I} |f_{\omega}(t)| ~dt;
\end{equation}
here  $I$ stands for any  interval in $\R$ and $|I|$ for its  Lebesgue measure. Recall that $f_{\omega}$ for almost all $\omega \in G$ is a locally integrable function on $\R$ with $f_{\omega}(t)=f(\omega\beta(t))$ for almost all $t\in \R$, and so $\overline{M}(f)(\omega)$ is defined almost everywhere.

\medskip

Note that the definition of $\overline{M}$ actually depends on the choice of the Dirichlet group $(G,\beta)$,
although for simplicity our notation does not indicate this. Moreover observe that \eqref{maximaloperatordef}  for  $(\T,\beta_{\T})$ precisely gives the classical Hardy-Littlewood maximal operator on $\T$ (see e.g. \cite[(3), p. 241]{Rudin}).
\begin{Theo} \label{HLoperator} Let $(G, \beta)$ be a Dirichlet group. Then the sublinear operator $\overline{M}$ is bounded from $L_{1}(G)$ to $L_{1,\infty}(G)$ and from $L_{p}(G)$ to $L_{p}(G)$, whenever $1<p\le \infty.$
\end{Theo}
In Section \ref{reductionsubsection} we see  that the proof of Theorem \ref{maximalineqtypeI} reduces to Theorem \ref{HLoperator}.
The following corollary is a  consequence of independent interest.
\begin{Coro} \label{differentiation} Let $f \in L_{1}(G)$. Then for almost all $\omega \in G$ we have
\begin{equation} \label{difff}
\lim_{T\to 0} \frac{1}{2T} \int_{-T}^{T} f_{\omega}(t) dt =f(\omega),
\end{equation}
and
\begin{equation} \label{besiconorm2}
\lim_{T\to \infty} \frac{1}{2T} \int_{-T}^{T} f_{\omega}(t)~ dt=\widehat{f}(0).
\end{equation}
\end{Coro}
Note that \eqref{difff} may be interpret as a 'differentiation theorem' for integrable functions on Dirichlet groups.
\medskip

Given $1 \leq p < \infty$, recall that the so-called Marcinkiewicz space $\mathcal{M}_{p}(\R)$ of all $f\in L_{loc}^{p}(\R)$ for which  $$\|f\|_{p}:=\limsup_{T\to \infty} \Big(\frac{1}{2T} \int_{-T}^{T} |f(t)|^{p} dt \Big)^{\frac{1}{p}}$$
exists, contains all trigonometric polynomials of the form $q(t):=\sum_{n=1}^{N} a_{n}e^{-itx_{n}}$ with  $x_{n}\in \R$ for $1 \leq n \leq N$, and in this case we have that
$$\|f\|_{p} = \lim_{T\to \infty}\Big(\frac{1}{2T} \int_{-T}^{T} |f(t)|^{p} dt \Big)^{\frac{1}{p}}$$
Then the Besicovitch space $\mathcal{B}_{p}(\R)$ is defined to be the closure of the trigonometric polynomials in $(\mathcal{M}_{p}(\R),\|\cdot\|_{p})$.
By density, we additionally have $\mathcal{B}_{p}(\R)=L_{p}(\overline{\R})$. Moreover, from \cite{DefantSchoolmann1} we know that $L_{p}(G)\subset L_{p}(\overline{\R})$ (isometrically) for all Dirichlet groups $(G,\beta)$, and so $L_{p}(G)\subset \mathcal{B}_{p}(\R)$. Now given $f\in L_{p}(G)$, it seems difficult to determine the corresponding function in $\mathcal{B}_{p}(\R)$. But we are going to deduce from $\eqref{besiconorm2}$ that at least   for almost all $\omega \in G$ the corresponding function for the translations $f(\omega \cdot)$ is given by $f_{\omega}$.
\begin{Rema}
Given $f \in L_{p}(G)$, then $f_{\omega}\in \mathcal{B}_{p}(\R)$ for almost all $\omega \in G$ and
\begin{equation*}
\|f_{\omega}\|_{\mathcal{B}_{p}(\R)}=\lim_{T\to \infty} \Big(\frac{1}{2T} \int_{-T}^{T} |f_{\omega}(t)|^{p}~ dt\Big)^{\frac{1}{p}}=\|f\|_{L_{p}(G)}.
\end{equation*}
\end{Rema}

All proofs of this section are given in Section \ref{HLsection} and Section \ref{mainproof}.

\medskip

\subsection{Norm approximation by first and second Riesz means} \label{approxnormchapter}
From \cite{DefantSchoolmann1} we know that for any frequency $\lambda$ the sequence $(h_{\lambda_{n}})$ forms a Schauder basis for $H_{p}^{\lambda}(G)$, provided $1 < p < \infty$. Note that this
can also be seen as  a straight forward consequence of the  maximal inequality  from  Theorem \ref{Dir-Duy}.

\medskip

Consequently for $1 < p < \infty$ every function from $H_{p}^{\lambda}(G)$ is approximated by its first and second Riesz means with respect to the norm. As shown by the next result, this for $p=1$ and  the first Riesz means is still true.

\begin{Theo} \label{fejer} Let  $f\in H_{1}^{\lambda}(G)$. Then for all $k>0$
\begin{equation} \label{Rieszmeandef}
\lim_{x\to \infty} R^{\lambda,k}_{x}(f)=f,
\end{equation}
the limit taken with respect to the $H_1$-norm.
\end{Theo}

But for the frequencies  $\lambda=(n)$ and $\lambda=(\log n)$ this result in general fails
if we replace first Riesz means by second Riesz means. This is  shown exactly as in Remark \ref{bardet} replacing almost everywhere convergence by convergence in the norm.

\begin{Rema} \label{approxsecondmeansnorm}
For  $\lambda=(n)$ and $\lambda=(\log n)$ the Fourier series of some $f\in H_{1}^{\lambda}(G)$ in general is not $(e^\lambda,1)$-Riesz summable in the $H_p$-norm, i.e.  the second $(\lambda,1)$-Riesz means  do not approximate the function in the  $H_p$-norm.
\end{Rema}

The proof of Theorem \ref{fejer} is given in Section \ref{fast}.

\medskip

\subsection{Uniformly almost periodic functions} \label{aap}
Finally, we give an application of Corollary \ref{Dirichletseries} (which follows from  our main  Theorem \ref{maximalineqtypeI}) to uniformly almost periodic functions. We show that $H_{\infty}^{\lambda}(G)$ may be identified    isometrically and 'coefficient preserving'  with the space of all bounded and holomorphic functions on $[Re>0]$, which  are uniformly almost periodic when restricted to any abscissa $[Re=\sigma], \sigma >0$.

\medskip

Before we state the result, let us recall a few basic definitions and facts from the theory of almost periodic functions -- we  refer to \cite[Chapter III]{Besicovitch} for more information.

\medskip

  A continuous function $g:\R \to \C$ is said to be uniformly almost periodic (see \cite[pp.1-2]{Besicovitch})
 if for  every $\varepsilon>0$ there is a number $l>0$ such that for all intervals $I\subset \R$ with $|I|=l$ there is  $\tau \in I$  such that
$$\sup_{ x \in \R} |g(x +\tau)-g(x)|<\varepsilon.$$
Equivalently, a continuous function $g\colon \R \to \C$ is uniformly almost periodic if and only if  it is the uniform limit of trigonometric polynomials of the form $p(t)=\sum_{n=1}^{N} a_{x_{n}} e^{-itx_{n}}$, where $x_{n} \in \R$ (see \cite[$2^\circ$ Theorem, p. 29]{Besicovitch}).

\medskip

A holomorphic  function $F$ defined on an open strip $[\alpha < Re < \beta]$ is said to be   uniformly almost periodic (see \cite[pp.141-142]{Besicovitch})
if for  every $\varepsilon>0$ there is a number $l>0$ such that for all intervals $I\subset \R$ with $|I|=l$ there is  $\tau \in I$  for which
$$\sup_{ z \in [\alpha < Re < \beta]} |F(z +i\tau)-F(z)|<\varepsilon.$$
Obviously, this implies that for every
$\alpha < \sigma < \beta$ each of its restrictions $F_{\sigma}=F(\sigma + i \cdot): \R \to \C$ is uniformly almost periodic.

\medskip

Moreover, in \cite[$4^\circ$ Theorem, p. 142-143]{Besicovitch} the following is proved: If a bounded and holomorphic function $F: [Re >0] \to \C$
for some $\sigma_0>0$
has a uniformly almost periodic restriction $F_{\sigma_{0}}=F(\sigma_0 + i \cdot): \R \to \C$, then it is uniformly almost periodic on every possible smaller strip
$[\alpha_1 <  Re < \beta_1]$  with $0 < \alpha_1 < \beta_1 < \infty$.

\medskip

\medskip Fixing $x \in \R$ and  $\sigma>0$, we call
$$a_{x}(F):=\lim_{T\to \infty} \frac{1}{2T} \int_{-T}^{T} F_{\sigma}(t) e^{-(\sigma+it)x} dt$$
the $x$th Bohr coefficient of $F$; it appears that the definition of $a_{x}$ is independent of the choice of $\sigma$ (see \cite[p. 147]{Besicovitch}). Moreover, it can be shown, that only at most countable many Bohr coefficients are non zero, and that $F$ vanishes, whenever its Bohr coefficients vanish (see \cite[p. 148 and p. 18]{Besicovitch}).
\medskip

 Note that finite sums of the form $F(z):=\sum_{n=1}^{N} a_{n} e^{-\lambda_{n}z}$ (all coefficients $\neq 0$) are typical examples of holomorphic, uniformly  almost periodic functions, and then the coefficients $a_{k},
  \, 1 \leq k \leq N,$ are precisely the (non-zero) Bohr coefficients of $F$.

\begin{Defi} Given a frequency $\lambda$, define  $\mathcal{H}_{\infty}^{\lambda}[Re>0]$ to be the space of all bounded, holomorphic functions $F\colon [Re>0] \to \C$, which are uniformly almost periodic on all (or equivalently some) abscissa $[Re=\sigma], \sigma>0$ and for which the Bohr coefficients $a_{x}(F)$ vanish whenever $x \notin \lambda$.
\end{Defi}
Together with $\|F\|_{\infty}:=\sup_{[Re>0]} |F(z)|$ the space $\mathcal{H}_{\infty}^{\lambda}[Re>0]$ becomes a Banach space, and we have $|a_{x}(F)|\le \|F\|_{\infty}$ for all $x \in \R$.

\medskip

Recall that, by \cite{DefantSchoolmann1}, if $\lambda$ satisfies $(LC)$, then the space $H_{\infty}^{\lambda}(G)$ equals $\mathcal{D}_{\infty}(\lambda)$, which is the space of all $\lambda$-Dirichlet series $D=\sum a_{n}e^{-\lambda_{n}s}$ converging and defining a bounded function on $[Re>0]$.  As proved in  \cite{Schoolmann1}, such Dirichlet series converge uniformly on ever half-plane $[Re>0], \sigma >0$, which implies  that their  limit functions  belong to $\mathcal{H}_{\infty}^{\lambda}[Re>0]$, i.e. under $(LC)$ the Banach space $H_{\infty}^{\lambda}(G)$ embeds isometrically into $\mathcal{H}_{\infty}^{\lambda}[Re>0]$. Using Theorem \ref{maximalineqtypeI}, we can show much more -- for any  $\lambda$ both spaces may be identified 'coefficient preserving'.

\medskip

\begin{Theo} \label{inftycaseTheo} Let $\lambda$ be an arbitrary  frequency and $(G,\beta)$ a $\lambda$-Dirichlet group. Then there is an onto isometry $$\Psi \colon \mathcal{H}_{\infty}^{\lambda}[Re>0] \to H_{\infty}^{\lambda}(G),~~ F \mapsto f,$$
which preserves  Bohr and Fourier coefficients.
\end{Theo}
Note that Theorem \ref{inftycaseTheo} again proves that the definition of $H_{\infty}^{\lambda}(G)$ is independent of the chosen $\lambda$-Dirichlet group $(G,\beta)$ (see also \cite{DefantSchoolmann1}).

\begin{Rema} Recall from \cite{Schoolmann1} that the space $\mathcal{D}^{ext}_{\infty}(\lambda)$ denotes the space of all somewhere convergent Dirichlet series, which allow a bounded holomorphic extension to $[Re>0]$. By \cite{DefantSchoolmann1}, for any frequency $\lambda$ there is an injective contraction $\Psi_{1} \colon \mathcal{D}^{ext}_{\infty}(\lambda)\hookrightarrow H_{\infty}^{\lambda}(G), ~~ D \mapsto f$, which preserves the Dirichlet and Fourier coefficients. On the other hand, by \cite{Schoolmann1} the map $\Psi_{2} \colon \mathcal{D}^{ext}_{\infty}(\lambda) \hookrightarrow \mathcal{H}_{\infty}^{\lambda}[Re>0], ~~ D \mapsto F,$ where $F$ is the limit function of $D$, defines an into isometry. Theorem \ref{inftycaseTheo} shows that $\Psi_{1}$ is even isometric and that the mapping $\Psi$ from Theorem \ref{inftycaseTheo} commutes with $\Psi_{1}$ and $\Psi_{2}$, that is we have $\Psi_{1}=\Psi\circ \Psi_{2}.$
\end{Rema}

\begin{Rema} To every function  $F\in \mathcal{H}_{\infty}^{\lambda}[Re>0]$ we may (formally) assign  its Dirichlet series $D:=\sum a_{\lambda_{n}}(F) e^{-\lambda_{n}s}$. The question appears, under what conditions on $\lambda$ does $D$ converge on $[Re>0]$. Recall that, if $\lambda$ satisfies Landau's condition (\ref{LC}), then $D$ even converges uniformly on $[Re>0]$ (see e.g. \cite{Schoolmann1} or \cite{Landau}).
\end{Rema}

We finish with the following consequence of Theorem \ref{inftycaseTheo}. By \cite{HLS} the Hardy space $H_{\infty}(\T^{\infty})$ isometrically equals  the Banach space $H_{\infty}(B_{c_{0}})$ of all bounded and holomorphic functions $F$ on the open unit ball  $B_{c_{0}}$  of $c_{0}$, identifying Fourier and Taylor coefficients (see \cite[Theorem 5.1]{Defant} for details). Moreover, given $F\in H_{\infty}(B_{c_{0}})$, for every $u>0$ the function
$$F_{u}(z):=F(\mathfrak{p}^{-u}z) \colon \T^{\infty} \to \C$$
is continuous and $\widehat{F_{u}}(\alpha)=c_{\alpha}(F)n^{-u}$ whenever $n=\mathfrak{p}^{\alpha}$. Hence, if $f\in H_{\infty}(\T^{\infty})$ and $F \in H_{\infty}(B_{c_{0}})$ are associated, then we have $f*p_{u}=F_{u}$, and so $f*p_{u} \in C(\T^{\infty})$  for every $u>0$. Theorem \ref{inftycaseTheo} extends this result to arbitrary $\lambda$-Dirichlet groups (compare also with Corollar~\ref{almosteverywheremaintheo} and Theorem~\ref{fejer}).
\begin{Coro} \label{shiftcontinuous} Let  $(G,\beta)$ be a $\lambda$-Dirichlet group, and $f \in H_{\infty}^{\lambda}(G)$. Then $f*p_{u} \in C(G)$ for all $u>0$, and for all $k>0$
$$\lim_{x\to \infty} R_{x}^{\lambda,k}(f*p_{u})=f*p_{u}$$
uniformly on $G$.
\end{Coro}

The proof of Theorem \ref{inftycaseTheo} is in Section \ref{fast}.

\section{\bf Proofs} \label{proofs}
The proofs of the results presented in the  previous sections are provided according to the following order:
\begin{itemize}
\item
Proof of Theorem \ref{HLoperator} (Section \ref{HLsection})
\item
Proof of Theorem \ref{maximalineqtypeI} (Section \ref{reductionsubsection})
\item
Proof of Corollary \ref{almosteverywheremaintheo}, Proposition~\ref{Fatouthm}, and Corollary \ref{differentiation} (Section \ref{mainproof})
\item
Proof of  Lemma \ref{hedesaks} and Corollary \ref{Dirichletseries} (Section \ref{mainproof2})
\item
Proof of the Theorems \ref{maxitranslatedtypeI}, \ref{maxitransinternaltype1}, and \ref{maxitranslatedtypeII} (Section \ref{translatedsubsection})
\item
Proof of the Theorems  \ref{fejer}, \ref{inftycaseTheo}, and  Corollary \ref{shiftcontinuous} (Section \ref{aap})
\end{itemize}

\subsection{Proof of Theorem \ref{HLoperator}} \label{HLsection}
Fixing $A>0$, consider the following 'graduated'
 variant of $\overline{M}$, that is
\begin{equation} \label{MbarA}
\overline{M}_{A}(f)(\omega):=\sup_{I \subset [-A,A]} \frac{1}{|I|} \int_{I} |f_{\omega}(y)| dy,
\end{equation}
where $f \in L_{1}(G)$.
Note that with this definition we have
\begin{equation} \label{friday}
\overline{M}(f)(\omega)=\sup_{I \subset \R} \frac{1}{|I|} \int_{I} |f_{\omega}(t)|~ dt=\sup_{A>0} \overline{M}_{A}(f)(\omega)=\sup_{N\in \N} \overline{M}_{N}(f)(\omega).
\end{equation}
Moreover, since for all intervals $I$ the function $$F(\omega,t):= \frac{1}{|I|}\chi_{I}(t)f(\omega \beta(t)) \colon G\times \R \to \R$$ is measurable, by Fubini's theorem for almost all $\omega \in G$ the function
$$t \mapsto \frac{1}{|I|} \int_{I} f_{\omega}(t) dt$$
is measurable. Recall that the pointwise supremum of a countable family of measurable functions is again measurable. So, if we in the definition from \eqref{MbarA} consider all intervals with rational boundary points, then $\overline{M}_{A}(f)$ is measurable and it leads to the same operator. Indeed, if $I$ is an non empty interval, then we are able to choose a sequence of subintervals $I_{n} \subset I$ with rational boundary points, such that $\lim_{n\to \infty}\big|I\setminus I_{n}\big|=0$. Then
\begin{align*}
&\Big|\frac{1}{|I_{n}|} \int_{I_{n}} |f_{\omega}(t)| dt-\frac{1}{|I|} \int_{I} |f_{\omega}(t)| dt\Big|\\ &=\Big|\frac{1}{|I_{n}|} \int_{I_{n}} |f_{\omega}(t)| dt-\frac{1}{|I|} \Big(\int_{I_{n}} |f_{\omega}(t)| dt+\int_{I\setminus I_{n}} |f_{\omega}(t)| dt\Big)\Big|\\ &
=\Big|\Big(\frac{1}{|I_{n}|}-\frac{1}{|I|}\Big)\int_{I_{n}}|f_{\omega}(t)| dt+ \frac{1}{|I|}\int_{I\setminus I_{n}} |f_{\omega}(t)| dt\Big|,
\end{align*}
and so tending $n\to \infty$ we obtain
$$\lim_{n\to \infty} \frac{1}{|I_{n}|} \int_{I_{n}} f_{\omega}(t) dt= \frac{1}{|I|} \int_{I} f_{\omega}(t) dt.$$
Hence $\overline{M}_{A}(f)$ is measurable for all $A>0$, and by $(\ref{friday})$ we see that $\overline{M}(f)$ is also measurable. Now Theorem \ref{HLoperator} is a consequence of the following lemma.

\begin{Lemm} \label{graduate}
Let $\alpha>0$ and $f\in L_{1}(G)$. Then for all $A>0$
$$m\Big(\left\{ \omega \in G \mid \overline{M}_{A}(f)(\omega)>\alpha\right\} \Big)\le \frac{10\|f\|_{1}}{\alpha}.$$
\end{Lemm}

\begin{proof}
Let $A>0$ and fix $\alpha>0$. We define
$$\Omega(\alpha):=\left\{ \omega \in G \mid \overline{M}_{A}(f)(\omega)>\alpha\right\}.$$
Then $\Omega(\alpha)$ is measurable, since $\overline{M}_{A}(f)$ is measurable. Moreover, for $\omega \in G$ we define
$$\Omega_{\omega}(\alpha):=\{t\in [-A,A] \mid \omega \beta(t)\in \Omega(\alpha)\},$$
which is Lebesgue-measurable for almost all $\omega \in G$, since $\Omega(\alpha)$ is measurable.
Hence by definition for all $t\in [-A,A]$
$$\omega \beta(t) \in \Omega(\alpha) \Leftrightarrow t \in \Omega_{\omega}(t),$$
and so we obtain for all $A>0$
\begin{align*}
m(\Omega(\alpha))&= \int_{G} \frac{1}{2A} \int_{-A}^{A} \chi_{\Omega(\alpha)}(\omega \beta(t)) dt dm
\\ &
= \int_{G} \frac{1}{2A} \int_{-A}^{A} \chi_{\Omega_{\omega}(\alpha)}(t) dt dm(\omega) =\int_{G} \frac{1}{2A}  \eta(\Omega_{\omega}(\alpha)) dm(\omega),
\end{align*}
where $\eta$ denotes the Lebesgue measure restricted to $[-A,A]$.
We now claim that
\begin{equation} \label{finalstep}
\eta(\Omega_{\omega}(\alpha))\le \frac{5}{\alpha}\int_{-2A}^{2A} |f_{\omega}(t)| dt.
\end{equation}
Indeed, if this estimate is verified, then we finally obtain
\begin{align*}
m(\Omega(\alpha))\le \frac{10}{\alpha} \int_{G} \frac{1}{4A} \int_{-2A}^{2A} |f_{\omega}(t)| dt dm(\omega)= \frac{10}{\alpha} \|f\|_{1}.
\end{align*}
So let us check \eqref{finalstep}. For every $t \in \Omega_{\omega}(\alpha)$ (by definition) there is an open interval $I_{t}\in [-2A,2A]$ containing $t$ such that
\begin{equation} \label{standardargument2}
\frac{1}{|I_{t}|} \int_{I_{t}} |f_{\omega}(y)| dy> \alpha.
\end{equation}
By Vitali's covering theorem (see e.g. \cite[Theorem 1.24, p. 36]{Evans}) there is a sequence $(t_{n}) \subset \Omega_{\omega}(\alpha)$ such that
$$\Omega_{\omega}(\alpha)\subset \bigcup_{t \in \Omega_{\alpha}} I_{t} \subset 5 \bigcup_{n \in \N} I_{t_{n}},$$
where the latter union is disjoint. So by $(\ref{standardargument2})$
\begin{equation*}
\eta(\Omega_{\omega}(\alpha))\le 5 \sum_{n=1}^{\infty} \eta\Big(I_{t_{n}}\Big)\le \frac{5}{\alpha} \int_{\bigcup_{n \in \N} I_{t_{n}}} |f_{\omega}(y)| dy \le \frac{5}{\alpha} \int_{-2A}^{2A} |f_{\omega}(y)| dy. \qedhere
\end{equation*}
\end{proof}

Finally, we are ready to give the

\begin{proof}[Proof of Theorem \ref{HLoperator}]
Take $f \in L_1(G)$ and $\alpha>0$, and define for $N\in \N$
$$\Omega(\alpha):=\{ \omega \in G \mid \overline{M}(f)(\omega)>\alpha\}, ~ \text{and}~~\Omega_{N}(\alpha):=\{ \omega \in G \mid \overline{M}_{N}(f)(\omega)>\alpha\}.$$
Then
$$\Omega(\alpha)=\bigcup_{N\in \N} \Omega_{N}(\alpha)$$
 and, since $\Omega_{N}(\alpha)\subset \Omega_{N+1}(\alpha)$ for all $N$, we by Lemma \ref{graduate} have
$$m(\Omega(\alpha))=\lim_{N\to \infty} m(\Omega_{N}(\alpha))\le \frac{10\|f\|_{1}}{\alpha}.$$
The case $p=\infty$ follows directly from the fact, that $\|f_{\omega}\|_{L_{\infty}(\R)}=\|f\|_{L_{\infty}(G)}$ (see \cite[Lemma 3.10]{DefantSchoolmann1}). Now Marcinkiewicz's interpolation theorem (see e.g. \cite[Theorem 1.3.2., p. 34]{Grafakos1}) gives the claim for $1<p<\infty$.
\end{proof}

\subsection{Proof of Theorem \ref{maximalineqtypeI}} \label{reductionsubsection}

The next proposition reduces the proof of Theorem~\ref{maximalineqtypeI} to Theorem~\ref{HLoperator}.
\begin{Prop} \label{reduction} Let $\lambda$ be a frequency and $(G, \beta)$  a $\lambda$-Dirichlet group. Then
 for every $k>0$ there is a constant $C(k)>0$ such for all $f \in H_{1}^{\lambda}(G)$ and  for almost all $\omega \in G$
\begin{equation} \label{reductioneq}
R_{\max}^{\lambda,k}(f)(\omega)=\sup_{x>0} \Big|R_{x}^{\lambda,k}(f)(\omega)\Big|\le C(k)\overline{M}(f)(\omega).
\end{equation}
Moreover, for $0<k\le 1$ the choice $C(k)=Ck^{-1}$ with an absolute constant $C$ is possible .
\end{Prop}

Before we start to prove this result we show how it gives the
\begin{proof}[Proof of Theorem \ref{maximalineqtypeI}]
Applying the $H_{p}^{\lambda}(G)$-norm to $(\ref{reductioneq})$, Theorem \ref{HLoperator} shows that $R_{\max}^{\lambda,k}$ defines a bounded operator from $H_{1}^{\lambda}(G)$ to $L_{1,\infty}(G)$, and from $H_{p}^{\lambda}(G)$ to $L_{p}^{\lambda}(G)$ whenever $1 < p < \infty$.
\end{proof}
We need two ingredients for the proof of Proposition \ref{reduction}. The first one is the following integral representation for first Riesz means, where we denote by $\mathcal{F}_{L_{1}(\R)}$ the Fourier transform on $L_{1}(\R)$.
\begin{Lemm} \label{lemma2type1}
Let $f \in H_{1}^{\lambda}(G)$ and $0<k\le1$. Then we for almost all $\omega \in G$ for all $x>0$ and $u\ge 0$ have
\begin{equation} \label{perroneq}
\Big(\frac{\Gamma(k+1)e}{2\pi x^{k} } \Big)^{-1}R^{\lambda,k}_{x}(f*p_{u})(\omega)= \int_{\R} f_{\omega}(a) \mathcal{F}_{L_{1}(\R)} \Big( \frac{P_{u+\frac{1}{x}}(\cdot-a)}{(\frac{1}{x}+i\cdot)^{1+k}} \Big)(-x) da.
\end{equation}
\end{Lemm}
In order to prove (\ref{reductioneq}), we will see, that in (\ref{perroneq}) it suffices to have control of the $L_{1}(\R)$-norm of the function the Fourier tranform is applied on; our second ingredient for the proof of Proposition \ref{reduction}.
\begin{Lemm} \label{lemma1type1} Let $v>0$, $0<k\le 1$ and $u\ge 0$. Then
\begin{equation} \label{lemma1type1ineq1}
\int_{\R} \frac{P_{u+v}(t-a)}{|v+it|^{1+k}} dt \le \Big(2+\frac{u}{v}\Big)^{\frac{1+k}{2}} \frac{1}{|u+ai|^{1+k}}.
\end{equation}
Moreover, if $u=0$, then
\begin{equation} \label{lemma1type1ineq2}
\int_{\R} \frac{P_{v}(t-a)}{|v+it|^{1+k}} dt\le \frac{2}{|v+ai|^{1+k}}.
\end{equation}
\end{Lemm}
We first show, how Lemma \ref{lemma2type1} and Lemma \ref{lemma1type1} imply Proposition \ref{reduction}. After that, we give a proof of Lemma \ref{lemma2type1} which uses Lemma \ref{lemma1type1}, and eventually we prove Lemma \ref{lemma1type1}.

\begin{proof}[Proof of Proposition \ref{reduction}] In a first step we assume that  $0<k\le 1$.
Let
\[
K(a):=\frac{1}{|1+ia|^{1+k}} \,\,\, \text{and}\,\,\, K_{x}(a):=x K(ax)=\frac{x}{|1+iax|^{1+k}}, x>0\,.
\]
Then with Lemma \ref{lemma2type1} and Lemma \ref{lemma1type1} we obtain
\begin{align*}
&\left(\frac{\Gamma(k+1)e}{2\pi}\right)^{-1}|R_{x}^{\lambda,k}(f)(\omega)|\le x^{-k} \int_{\R} |f_{\omega}(a)| \left\|\frac{P_{\frac{1}{x}}(\cdot-a)}{(\frac{1}{x}+i\cdot)^{1+k}}\right\|_{L_{1}(\R)} da \\ &  \le x^{-k}\int_{\R}|f_{\omega}(a)| \frac{2}{|\frac{1}{x}+ia|^{1+k}} ~da = 2 \int_{\R}|f_{\omega}(a)| K_{x}(a) ~da \\ & =2\sup_{x>0} (|f_{\omega}|*K_{x})(0).
\end{align*}
Now by \cite[Theorem 2.1.10, p.91]{Grafakos1} we have
$$\sup_{x>0} |f_{\omega}|*K_{x}(0)\le \|K\|_{L_{1}(\R)} \sup_{T>0} \frac{1}{2T} \int_{-T}^{T} |f_{\omega}(t)| dt \le \|K\|_{L_{1}(\R)} \overline{M}(f)(\omega),$$
which proves the claim with constant $$C(k)=\frac{e}{\pi} \frac{\Gamma(k+1)}{k}.$$
If $k>1$, we write $k=l+k^{\prime}$, where $l \in \N$ and $k^ {\prime}\in ]0,1]$, and use  for every $x>0$ the following identity from \cite [Lemma 6, p. 27]{HardyRiesz}:
\begin{equation} \label{fromhightolow}
R_{x}^{\lambda,k}(f)(\omega)=\frac{\Gamma(k+1)}{\Gamma(l)\Gamma(k^{\prime}+1)} x^{-k}
\int_{0}^{x} R_{t}^{\lambda,k^{\prime}}(f)(\omega)  t^{k^{\prime}} (x-t)^{l-1} dt,
\end{equation}
where by substitution (see \cite[p. 27]{HardyRiesz})
\begin{equation*} \label{gammaidentity}
\frac{\Gamma(l)\Gamma(k^{\prime}+1)}{\Gamma(k+1)}=x^{-k} \int_{0}^{x} t^{k^{\prime}} (x-t)^{l-1} dt.
\end{equation*}
Together this leads to
\begin{equation} \label{gammaidentity}
|R_{x}^{\lambda,k}(f)(\omega)| \leq \sup_{0<t<x} |R_{t}^{\lambda,k'}(f)(\omega)|\,,
\end{equation}
which, applying the first step with $0 < k' \leq 1$, finishes the proof.
\end{proof}

\begin{proof}[Proof of Lemma \ref{lemma2type1}] Fix $u\ge 0$ and let first $f=\sum_{n=1}^{N} a_{n} h_{\lambda_{n}}$. Then $f_{\omega}(t)=\sum_{n=1}^{N} a_{n} h_{\lambda_n}(\omega) e^{-it\lambda_{n}}$ for all $\omega \in G$ and, since for all $\alpha>0$ and $k>0$
\begin{equation}\label{genius}
\frac{\Gamma(k+1)}{2\pi i}\int_{\alpha-i\infty}^{\alpha+i\infty} \frac{e^{ys}}{s^{1+k}} ds = \begin{cases} y^{k}&, \text{if } y\ge 0\\ 0 &, \text{if } y<0 \end{cases}
\end{equation}
(see e.g. \cite[Lemma 10, p. 50]{HardyRiesz}),  we for all $c>u$, $x>0$ obtain (with $\alpha=c-u$)
\begin{align*}
&\Big(\frac{\Gamma(k+1)}{2\pi x^{k}}\Big)^{-1}\sum_{\lambda_{n}<x} a_{n}e^{-u\lambda_{n}} h_{\lambda_n}(\omega) \Big(1-\frac{\lambda_{n}}{x}\Big)^{k}\\ &= \int_{\R} \frac{\sum_{n=1}^{N} a_{n}h_{\lambda_n}(\omega)e^{-\lambda_{n}(c+it)}}{(c-u+it)^{1+k}} e^{x(c-u+it)} dt\\ &= e^{x(c-u)}\int_{\R} \frac{(f_{\omega}*P_{c})(t)}{(c-u+it)^{1+k}} e^{ixt} dt=e^{x(c-u)}\int_{\R} f_{\omega}(a) \int_{\R} \frac{P_{c}(t-a)}{(c-u+it)^{1+k}}e^{ixt} dt ~ da.
\end{align*}
The choice $c=\frac{1}{x}+u$ leads to
\begin{equation} \label{fourier1}
R_{x}^{\lambda,k}(f*p_{u})(\omega)=\frac{\Gamma(k+1)e}{2\pi x^{k}} \int_{\R} f_{\omega}(a) \mathcal{F}_{L_{1}(\R)} \Big( \frac{P_{\frac{1}{x}+u}(\cdot-a)}{(\frac{1}{x}+i\cdot)^{1+k}} \Big)(-x) da,
\end{equation}
and so the claim holds for polynomials in $H_{1}^{\lambda}(G)$.
To proof the general case, observe that for all $u\ge0$ and $v>0$ the operator
\begin{equation} \label{operatorA}
\mathcal{A} \colon L_{1}(G) \to L_{1}(G,L_{1}(\R)), ~~ f \mapsto \left[\omega \to \frac{f_{\omega}*P_{u+v}}{(v+i\cdot)^{1+k}}\right]
\end{equation}
is bounded. Indeed, by Lemma \ref{lemma1type1} (and Fubini's theorem) we have
\begin{align*}
&\|\mathcal{A}(f)\|=\int_{G} \Big| \int_{\R} \frac{f_{\omega}*P_{u+v}(y)}{\Big(v+iy\Big)^{1+k}} dy \Big| d\omega \le \int_{G} \int_{\R} \int_{\R} |f_{\omega}(t)| \frac{P_{u+v}(y-t)}{|v+iy|^{1+k}} dt dy d\omega \\ & = \int_{G} \int_{\R} |f_{\omega}(t)| \int_{\R} \frac{P_{u+v}(y-t)}{|v+iy|^{1+k}} dy dt d\omega  \le \int_{G} \int_{\R} |f_{\omega}(t)| \frac{C(u,v,k)}{|u+it|^{1+k}} dt d\omega \\ & =C(u,v,k)\int_{\R} \int_{G} |f_{\omega}(t)| \frac{1}{|u+it|^{1+k}} d\omega dt = C(u,v,k) \int_{\R} \frac{1}{|u+it|^{1+k}} dt \|f\|_{1}\\ &=C_{1}(u,v,k)\|f\|_{1}.
\end{align*}
Additionally, this shows, that  $\frac{f_{\omega}*P_{u+v}}{(v+i\cdot)^{1+k}} \in L_{1}(\R)$ for almost all $\omega \in G$, and so we in particular obtain (with Fubini's theorem)
\begin{equation} \label{Fourier2}
\mathcal{F}_{L_{1}(\R)}\Big( \frac{f_{\omega}*P_{\frac{1}{x}+u}}{\Big(\frac{1}{x}+i\cdot\Big)^{1+k}}\Big)(-x)=\int_{\R} f_{\omega}(a) \mathcal{F}_{L_{1}(\R)} \Big( \frac{P_{\frac{1}{x}+u}(\cdot-a)}{(\frac{1}{x}+i\cdot)^{1+k}} \Big)(-x) da.
\end{equation}
Now let $(Q^{n})\subset H_{1}^{\lambda}(G)$ be a sequence of polynomials converging to $f$ in $H_{1}^{\lambda}(G)$ (see \cite{DefantSchoolmann1}). Then, by continuity of $\mathcal{A}$ and $\mathcal{F}_{L_{1}(\R)}$, we for almost all $\omega \in G$ obtain
$$\mathcal{F}_{L_{1}(\R)}\Big( \frac{f_{\omega}*P_{\frac{1}{x}+u}}{\Big(\frac{1}{x}+i\cdot\Big)^{1+k}}\Big)=\lim_{k\to \infty}\mathcal{F}_{L_{1}(\R)}\Big( \frac{Q_{\omega}^{n_{k}}*P_{\frac{1}{x}+u}}{\Big(\frac{1}{x}+i\cdot\Big)^{1+k}}\Big),$$
for some subsequence $(Q^{n_{k}})$, with uniformly convergence on $\R$. So together with (\ref{fourier1}) and (\ref{Fourier2})
\begin{align*}
&\mathcal{F}_{L_{1}(\R)}\Big( \frac{f_{\omega}*P_{\frac{1}{x}+u}}{\Big(\frac{1}{x}+i\cdot\Big)^{1+k}}\Big)(-x)=\lim_{k\to \infty}\mathcal{F}_{L_{1}(\R)}\Big( \frac{Q_{\omega}^{n_{k}}*P_{\frac{1}{x}+u}}{\Big(\frac{1}{x}+i\cdot\Big)^{1+k}}\Big)(-x)\\ &=\lim_{k\to \infty} R_{x}^{\lambda,k}(Q^{n_{k}}*p_{u})(\omega)=R_{x}^{\lambda,k}(f*p_{u})(\omega),
\end{align*}
which gives the claim by (\ref{Fourier2}).
\end{proof}

\begin{proof}[Proof of Lemma \ref{lemma1type1}]
By substitution we have
\begin{equation} \label{calcusubst}
\int_{-\infty}^{\infty} \frac{P_{u+v}(t-a)}{|v+it|^{1+k}} dt =\int_{\R} \frac{1}{|v+i(y+a)|^{1+k}} P_{u+v}(y) dy.
\end{equation}
We interpret the right hand side of $(\ref{calcusubst})$ as the $L_{1+k}$-norm with respect to the measure $d\mu= P_{u+v} d\lambda$, where $d\lambda$ denotes the Lebesgue measure on $\R$. Since $(\R,d\mu)$ is a finite measure space, we for all $0\le k\le 1$ have
$$\left\|\frac{1}{v+i(\cdot+a)}\right\|^{1+k}_{L_{1+k}(\R,d\mu)} \le \left\|\frac{1}{v+i(\cdot+a)}\right\|^{1+k}_{L_{2}(\R,d\mu)}.$$
Hence, it suffices to determine $(\ref{calcusubst})$ for $k=1$.
In this case, a straight calculation gives
$$\int_{\R} \frac{P_{u+v}(t-a)}{v^{2}+t^{2}} dt=\frac{1}{v} \frac{(u+2v)}{(u+2v)^{2}+a^{2}}.$$
So, if $u=0$, then we have
$$\int_{\R} \frac{P_{u+v}(t-a)}{v^{2}+t^{2}} dt=\frac{2}{4v^{2}+a^{2}}\le \frac{2}{v^{2}+a^{2}},$$
and, if $u\ne 0$, then we estimate
\begin{align*}
\int_{\R} \frac{P_{u+v}(t-a)}{v^{2}+t^{2}} dt \le  \frac{(u+2v)}{v} \frac{1}{u^{2}+a^{2}}= \frac{2+\frac{u}{v}}{|u+ai|^{2}}.
\end{align*}
Now by taking the $\frac{1+k}{2}$th power the claims follow.
\end{proof}

\subsection{Proof of Corollary \ref{almosteverywheremaintheo}, Proposition~\ref{Fatouthm}, and Corollary \ref{differentiation}} \label{mainproof}
It is quite standard to deduce almost everywhere convergence from appropriate maximal inequalities. Nevertheless we  for the sake of completeness  add a few details. We use the following standard
consequence of Egoroff's theorem (see e.g. \cite[Theorem 4.4, p. 33]{Stein}).

\begin{Rema} \label{egoroff} Let  $f_{n}: \Omega \to \C$ be measurable functions on a finite measure space
$(\Omega, \mu)$. Then $(f_{n})$ converges to $0$  almost everywhere if and only if for every $\varepsilon>0$ we have
$$\mu\Big(\Big\{x \in A \mid \limsup_{n\to \infty} |f_{n}(x)|\ge \varepsilon\Big\} \Big)=0.$$
\end{Rema}

The following device adapts some well-known arguments to our special situation.

\begin{Lemm} \label{egoroff2}
Let $X$ be a subspace of $L_{1,\infty}(\mu)$, and
$\big(T_{x,y}\colon X \to L_{1}(\mu)\big)_{x,y>0}$ and
$\big(S_y\colon Y\to L_{1,\infty}(\mu)\big)_{y>0}$
two families of  linear operators such that the sublinear maps
\begin{align*}
&
T\colon X \to L_{1,\infty}(\mu), ~~ f \mapsto \Big[w \mapsto \sup_{x,y>0} |T_{x,y}(f)(w)|\Big]
\\&
S\colon X \to L_{1,\infty}(\mu), ~~ f \mapsto \Big[w \mapsto \sup_{y>0} |S_{y}(f)(w)|\Big]
\end{align*}
are  bounded. Moreover, let $Y$ be a dense subset of $X$ such that for all $g\in Y$
\begin{equation*} \label{assumption}
\lim_{x\to \infty} \sup_y \big| T_{x,y}(g)-S_y(g)\big| =0\,\,\, \text{almost everywhere}\,.
\end{equation*}
Then this equation even holds for all  $f\in X$.
\end{Lemm}

\begin{proof}
Let $f\in X$ and $\varepsilon>0$. According to Remark \ref{egoroff} we show that
\begin{equation} \label{johanna}
m\Big( \Big\{ \omega \in G \mid \limsup_{x\to \infty} \sup_{y} |T_{x,y}(f)(\omega)-S_y(\omega)| \ge \varepsilon \Big\} \Big)=0\,.
\end{equation}
Denote the set which appears on the left side by $\Omega$, and use  the assumption on $Y$  to conclude for  all $\omega \in \Omega$ and $g \in Y$
\begin{align*}
\varepsilon &\le \limsup_{x\to \infty} \sup_{y} |T_{x,y}(f)(\omega)-S_y(f)(\omega)|\\ &\le \limsup_{x\to \infty} \big[\sup_{y}|T_{x,y}(f-g)(\omega)|+ \sup_{y}|S_y(f-g)(\omega)|\big]\\ &\le
|T(f-g)(\omega)|+ \sup_y|S_y(f-g)(\omega)|.
\end{align*}
Hence  the boundedness of $T$ and $S_y$ (and the quasi triangle inequality in $L_{1,\infty}(\mu)$) gives  some $C >0$ such that for all  $g\in Y$
$$\varepsilon m(\Omega) = \| \varepsilon \chi_{\Omega} \|_{1,\infty} \leq   2\|T(f-g)\|_{1,\infty}+2\|f-g\|_{1,\infty}\le 4 C\|f-g\|_1\,.$$
Finally, the density of $Y$ in $X$ gives the conclusion.
\end{proof}

We will also need  the following 'shifted' version of Theorem \ref{maximalineqtypeI}.
\begin{Prop} \label{shiftversion}
Let $f \in H_{1}^{\lambda}(G)$ and $k>0$. Then
$$\left\| \sup_{u\ge 0} R_{\max}^{\lambda,k}(f*p_{u}) (\cdot)\right\|_{1,\infty}\le C(k) \|f\|_{1}.$$
\end{Prop}
\begin{proof}
By Proposition \ref{reduction} and Theorem \ref{maximalineqtypeI}, it suffices to show, that
$$\left\|\sup_{u>0}\overline{M}(f*p_{u})(\cdot)\right\|_{1,\infty}\le \|\overline{M}(f)\|_{1,\infty}.$$
Indeed, for all intervals $I\subset \R$ and $u>0$ we have (using Fubini's theorem)
\begin{align*}
&\frac{1}{|I|} \int_{I} |(f_{\omega}*P_{u})(t)|~dt=\frac{1}{|I|} \int_{I} \Big| \int_{\R} f_{\omega}(t-a) P_{u}(a)~da \Big|~ dt \\ &\le \frac{1}{|I|} \int_{I}\int_{\R} |f_{\omega}(t-a)| P_{u}(a)~da~ dt= \int_{\R} P_{u}(a) \Big( \frac{1}{|I|} \int_{I} |f_{\omega}(t-a)|~dt\Big) ~da \\ &\le \int_{\R} P_{u}(a) \overline{M}(f)(\omega)~da\le \overline{M}(f)(\omega),
\end{align*}
and so, since the 'restriction' of $f*p_{u}$ to $\R$ is given by the function $f_{\omega}*P_{u}$, we for almost all $\omega \in G$ have
$$\sup_{u>0} \overline{M}(f*p_{u})(\omega)\le \overline{M}(f)(\omega).$$
Applying the $L_{1,\infty}$-norm gives the  inequality we claimed.
\end{proof}

\begin{proof}[Proof of Corollary \ref{almosteverywheremaintheo}]
Let us first proof  \eqref{conv}. Take $X=H_{1}^{\lambda}(G)$, $T_{x,y}(f):=R_{x}^{\lambda,k}(f)$ and $S_y$ the identity map. Then $\lim_{x\to \infty}R_{x}^{\lambda,k}(P)=P$ pointwise for all polynomials from $H_{1}^{\lambda}(G)$, and so the claim follows from Lemma \ref{egoroff2} and Theorem \ref{maximalineqtypeI}.
The proof of \eqref{conv2} is similar and needs Proposition \ref{shiftversion}.
\end{proof}

\begin{proof}[Proof of Proposition~\ref{Fatouthm}]
We first show, that there is a null set $N\subset G$ such that for all $\omega \notin N$ the integral
$$f*p_{u}(\omega)=\int_{\R} f_{\omega}(t) P_{u}(t) dt=f_{\omega}*P_{u}(0)$$
is finite for all $u>0$. Indeed, recall from Section \ref{intro} that  $f_{\omega}$ is  locally Lebesque-integrable  for almost all $\omega$, and so by \cite[Theorem 2.1.10, p. 91]{Grafakos1} we for all $u>0$
and almost all $\omega$ have
$$|f_{\omega}*P_{u}(0)|\le |f_{\omega}|*P_{u}(0)\le \|P_{1}\|_{1} \overline{M}(f)(\omega)=\overline{M}(f)(\omega)\,.$$
Since $\overline{M}(f)\in L_{1,\infty}(G)$ by Theorem \ref{HLoperator}, we obtain that $ \overline{M}(f)(\omega)<\infty$ almost everywhere and that the operator $T$ is defined. Moreover, $|T(f)(\omega)|\le \overline{M}(f)(\omega)$ for almost all $\omega$, and again Theorem \ref{HLoperator} implies that $T$ is bounded from $L_{1}(G)$ to $L_{1,\infty}(G)$ and $L_{p}(G)$ to $L_{p}(G)$, whenever  $1<p\le \infty$. The 'in particular' is then a consequence of Lemma \ref{egoroff2} with the choice $T_{x}(f)=f*p_{x}$, $S_{y}\equiv id$, and $Y$  the set of all polynomials.
\end{proof}

\begin{proof}[Proof of Corollary \ref{differentiation}]
Equations (\ref{difff}) and (\ref{besiconorm2}) are checked straight forward on polynomials. Then both claims follow from Lemma \ref{egoroff2} and Theorem \ref{HLoperator} by choosing $X=L_{1}(G)$ and $T_{x,y}(f)=\frac{1}{2x} \int_{-x}^{x} f_{\omega}(t) dt$  (resp. $\widetilde{T}_{x,y}(f):=T_{x^{-1},y}(f)$), since clearly $T_{x,y}(f)(\omega)\le \overline{M}(f)(\omega)$ for all $x,y$.
\end{proof}

\subsection{Proof of  Lemma \ref{hedesaks} and Corollary \ref{Dirichletseries}} \label{mainproof2}
We start with the proof of Lemma \ref{hedesaks}, which shows how Riesz-summability of the Fourier series of a  function $f\in H_{p}^{\lambda}(G)$ transfers to summability of the vertical limit of $D^{\omega}$, where $D:=\Bcal(f)$, and vice versa.

\begin{proof}[Proof of Lemma \ref{hedesaks}]
It suffices to check that $(1)$ and $(2)$ are equivalent.
Given a measurable set $A \subset G$ for almost every $\omega \in G$ the set
$$A_{\omega}:=\{ t \in \R \mid \beta(t)\omega \in A \}$$
is Lebesgue-measurable and by Fubini's theorem we have
$$m(A)=\int_{G} \int_{\R} \chi_{A}(\omega \beta(t)) \frac{1}{1+t^{2}} dt d\omega=\int_{G} \widetilde{\lambda}(A_{\omega}) d\omega,$$
where $\widetilde{\lambda}=(1+t^{2})^{-1} dt.$
Hence if, given $\varepsilon>0$, we define
$$\Omega:=\left\{ \omega \in G \mid \limsup_{n\to \infty} |f_n(\omega)-f(\omega)|\ge \varepsilon\right\},$$
then
$$\Omega_{\omega}=\left\{ t \in \R \mid \limsup_{n\to \infty} |(f_n)_\omega(t)-f_{\omega}(t)|\ge \varepsilon\right\}$$
and so
\begin{equation} \label{bayart}
m( \Omega )=\int_{G} \widetilde{\lambda}(\Omega_{\omega}) ~dm(\omega).
\end{equation}
By Remark \ref{egoroff}, assuming $(1)$, the left hand side of \eqref{bayart} vanishes, and so for almost all $\omega \in G$ we  for almost all $t\in \R$ have
$$\lim_{x\to \infty} \frac{1}{1+t^{2}} (f_n)_\omega(t)=f_{\omega}(t) \frac{1}{1+t^{2}},$$
and so equivalently
$$\lim_{x\to \infty} (f_n)_\omega(t)=f_{\omega}(t).$$
Vice versa, assuming (2), the right hand side of \eqref{bayart} vanishes, and so (1) follows from Remark \ref{egoroff}.
\end{proof}

\begin{proof}[Proof of Corollary \ref{Dirichletseries}]
Translate Corollary \ref{almosteverywheremaintheo} with the help of   Lemma \ref{hedesaks} into Dirichlet series.
\end{proof}

\subsection{Proof of the Theorems \ref{maxitranslatedtypeI}, \ref{maxitransinternaltype1}
and \ref{maxitranslatedtypeII}}
 \label{translatedsubsection}
For the proof of Theorem \ref{maxitranslatedtypeI} we need the  following
\begin{Lemm} \label{Abeltype1} Let $0<k\le 1$ and $u>0$. Then for every $\varepsilon>0$ there is a constant $C=(k,u,\varepsilon)$ such that for all $x>0$ and complex sequences $(a_{n})$ we have
\begin{equation*}
\Big|\sum_{\lambda_{n}<x} a_{n}e^{-(u+\varepsilon)\lambda_{n}} \Big(1-\frac{\lambda_{n}}{x}\Big)^{k}\Big|\le C \sup_{0<y\leq x} \Big| e^{-uy} \sum_{\lambda_{n}<y} a_{n}\Big(1-\frac{\lambda_{n}}{y}\Big)^{k} \Big|.
\end{equation*}
\end{Lemm}
The proof of  Lemma \ref{Abeltype1} follows  from a careful analysis of Theorem 24 from Hardy and M. Riesz \cite[\S VI.3, p. 42 ]{HardyRiesz}. Among others, we use the following identity from \cite[\S IV.2, p. 21]{HardyRiesz}:
\begin{equation} \label{idfirst}
\sum_{\lambda_{n}<x} a_{n}\Big(1-\frac{\lambda_{n}}{x}\Big)^{k}=k
x^{-k} \int_{0}^{x} \Big( \sum_{\lambda_{n}<t} a_{n}\Big)(x-t)^{k-1} dt.
\end{equation}
 Moreover, in the case $0<k<1$, we need the  following two integrals
\begin{equation} \label{gammasubst}
\int_{y}^{x}(t-y)^{-k} (x-t)^{k-1}dt=\Gamma(1-k)\Gamma(k)
\end{equation}
and
\begin{equation} \label{gammasubst2}
\int_{y}^{\infty} (t-y)^{-k} e^{-(u+\varepsilon) t} dt=(u+\varepsilon)^{k-1}e^{-(u+\varepsilon)y}\Gamma(1-k);
\end{equation}
 the first follows by simple substitution using the beta function and the second one is obvious.

\begin{proof}[Proof of Lemma \ref{Abeltype1}]
Let us write for simplicity $$R^{\lambda,k}_{x}:=\sum_{\lambda_{n}<x} a_{n}\Big(1-\frac{\lambda_{n}}{x}\Big)^{k} \text{ and  }~~ \Delta:=\Delta(x,u,k):=\sup_{0<y<x} e^{-yu} |R_{y}^{\lambda,k}|.$$
Then, defining $h(t):=(e^{-(u+\varepsilon)t}-e^{-(u+\varepsilon)x})(x-t)^{k}$ for $0 < t <x$,  we obtain
\begin{align*}
&\sum_{\lambda_{n}<x} a_{n}e^{-(u+\varepsilon)\lambda_{n}}(x-\lambda_{n})^{k}
= -\int_0^x \Big(\sum_{\lambda_n < t} a_n \Big)\,\frac{d}{dt} \Big(  e^{-(u+\varepsilon)t} (x-t)^{k}  \Big) dt\\ &=ke^{-(u+\varepsilon)x} \int_{0}^{x} \Big(\sum_{\lambda_{n}<t} a_{n}\Big)(x-t)^{k-1} dt- \int_{0}^{x} \Big(\sum_{\lambda_{n}<t} a_{n}\Big) h^{\prime}(t) dt\\ &
=e^{-(u+\varepsilon)x}x^{k}R^{\lambda,k}_{x}+ \int_{0}^{x} \Big(\sum_{\lambda_{n}<t} a_{n}(t-\lambda_{n})^{1}\Big) h^{\prime \prime}(t) dt=:A+B,
\end{align*}
where the first equality follows from   Abel summation (see \cite[p. 40]{HardyRiesz}), and the third by \eqref{idfirst} and  partial integration, since
\begin{equation} \label{capri}
\frac{d}{dt}\big(\sum_{\lambda_{n}<t} a_{n}(t-\lambda_{n})^{1}\big)=\sum_{\lambda_{n}<t}a_{n}.
\end{equation}
Now let first $0<k<1$. Then by \cite[Lemma 6, p.27]{HardyRiesz} we have
\begin{equation}\label{motor}
\sum_{\lambda_{n}<t} a_{n}(t-\lambda_{n})^{1}=\frac{1}{\Gamma(1+k)\Gamma(1-k)}\int_{0}^{t} \Big(\sum_{\lambda_{n}<y} a_{n}(y-\lambda_{n})^{k}\Big)(t-y)^{-k} dy.
\end{equation}
 Then Fubini's theorem implies
\begin{align*}
B&=C(k)\int_{0}^{x} \int_{0}^{t} \Big(\sum_{\lambda_{n}<y} a_{n}(y-\lambda_{n})^{k}\Big)(t-y)^{-k} dy h^{\prime \prime}(t) dt\\ &=C(k) \int_{0}^{x}y^{k}R^{\lambda,k}_{y} \int_{y}^{x} (t-y)^{-k} h^{\prime \prime}(t) dt dy,
\end{align*}
where
\begin{align*}
h^{\prime \prime}(t)
=
(u+\varepsilon)^{2}e^{-(u+\varepsilon)t} (x-t)^{k}
&
+2k(u+\varepsilon)e^{-(u+\varepsilon)t}(x-t)^{k-1}\\ &+k(k-1)(e^{-(u+\varepsilon)t}-e^{-(u+\varepsilon)x})(x-t)^{k-2}.
\end{align*}
Using $(e^{-(u+\varepsilon)t}-e^{-(u+\varepsilon)x})\le e^{-(u+\varepsilon)t}(u+\varepsilon)(x-t)$ for the third summand we estimate
\begin{equation} \label{bound}
|h^{\prime \prime}(t)|\le C_{1}(u,k, \varepsilon)e^{-(u+\varepsilon)t}\Big((x-t)^{k}+(x-t)^{k-1}\Big).
\end{equation}
Then we deduce from  \eqref{gammasubst} and  \eqref{gammasubst2} that for all $y>0$
\begin{align*}
|B|
&
\le C_{2}  \int_{0}^{x}y^{k} |R_{y}^{\lambda,k}| \int_{y}^{x} e^{-(u+\varepsilon)t} \left(x^{k}+(x-t)^{k-1}\right)(t-y)^{-k} dt dy \\ &\le C_{3}  \int_{0}^{x} y^{k} |R_{y}^{\lambda,k}| \left( x^{k} e^{-(u+\varepsilon)y} + e^{-(u+\varepsilon)y}\int_{y}^{x}(x-t)^{k-1}(t-y)^{-k} dt \right) dy\\ &\le C_{4}\int_{0}^{x}y^{k} |R_{y}^{\lambda,k}| e^{-(u+\varepsilon)y}\left(x^{k}+1)\right) dy \le C_{4} x^{k} \Delta  \int_{0}^{x} y^{k} e^{-\varepsilon y} (1+x^{-k}) dy \\& \le C_{4}x^{k} \Delta \left( \int_{0}^{x} y^{k} e^{-\varepsilon y} dy + \int_{0}^{x} e^{-\varepsilon y} dy \right) \le C_{5} x^{k} \Delta.
\end{align*}
Hence finally
\begin{equation*}
\Big|\sum_{\lambda_{n}<x} a_{n}e^{-(u+\varepsilon)\lambda_{n}}\Big(1-\frac{\lambda_{n}}{x}\Big)^{k}\Big|\le x^{-k}(|A|+|B|)\le C_{6}  \sup_{0<y\leq x} e^{-yu} |R_{y}^{\lambda,k}|.
\end{equation*}
Note, that the case $k=1$ follows the same lines with the difference, that we do not use (\ref{motor}) and estimate $|B|$ directly.
\end{proof}

\begin{proof}[Proof of Theorem \ref{maxitranslatedtypeI}] Fix $k,u >0$, and assume first that
 $0<k\le 1$. Then by Lemma \ref{Abeltype1} it suffices to prove that for all   $f \in H_p^\lambda(G)$
\begin{equation} \label{meet}
\Big\|\sup_{x>0} |e^{-ux}R_{x}^{\lambda,k}(f)(\cdot)|\Big\|_{p}\le C(k)\|f\|_{p}.
\end{equation}
Let first $f=\sum_{n=1}^{N} a_{n} h_{\lambda_{n}}$ be a polynomial. Then applying \cite[Lemma 3.6]{Schoolmann1}
(with $\varepsilon=u$ and $D^{\omega}=\sum_{n=1}^{N} a_{n}h_{\lambda_n}(\omega) h_{\lambda_{n}}$, or using again \eqref{genius} straight away) we obtain for all $x>0$ and $\omega \in G$
\begin{equation} \label{balkon}
e^{-u x} R_{x}^{\lambda,k}(f)(\omega)=\frac{\Gamma(k+1)}{2\pi} \mathcal{F}_{L_{1}(\R)} \Big(\frac{f_{\omega}*P_{u}}{(u+i\cdot)^{1+k}}\Big)(-x).
\end{equation}
Like in the proof of Lemma \ref{lemma2type1} the continuity of $\mathcal{A}$ from (\ref{operatorA}) as well as  the continuity of  the Fourier transform $\mathcal{F}_{L_{1}(\R)}$ imply that \eqref{balkon} holds
for every
$f\in H_{1}^{\lambda}(G)$, all $x>0$, and  almost all $\omega \in G$. Hence for such $f,x$ and $\omega$
\begin{align*}
&
\left(\frac{\Gamma(k+1)}{2\pi} \right)^{-1}\Big|e^{-u x}  R_{x}^{\lambda,k}(f)(\omega)\Big|
=\Big|\mathcal{F}_{L_{1}(\R)} \Big(\frac{f_{\omega}*P_{u}}{(u+i\cdot)^{1+k}}\Big)(-x)\Big|
\\&
=\Big|\int_{\R} \frac{f_{\omega}*P_{u}(t)}{(u+it)^{1+k}}e^{itx} dt\Big| \le  \int_{\R} |f_{\omega}(a)| \left\|\frac{P_{u}(\cdot-a)}{(u+i\cdot)^{1+k}}\right\|_{L_{1}(\R)} \le 2\int_{\R} \frac{|f_{\omega}(a)|}{|u+ia|^{1+k}} da,
\end{align*} where we used Lemma \ref{lemma1type1} for the last inequality. Now integration over $G$ and the Minkowski inequality give for all $f \in H_p^\lambda(G)$
\begin{equation*} \label{computer}
\left(\int_{G} \sup_{x>0} |e^{-ux}R_{x}^{\lambda,k}(f)(\omega)|^{p} d\omega \right)^{\frac{1}{p}}\le \frac{1}{\pi} \frac{\Gamma(k+1)}{k}\|f\|_{p},
\end{equation*}
which under the restriction that  $0<k \leq 1$ is what we aimed at in \eqref{meet}.
If $k>1$, then we write $k=l+k^{\prime}$, where $l\in \N$ and $0<k^{\prime}\le 1$. Replacing $f$ by $f*p_{u}$ in \eqref{gammaidentity} we conclude that
\begin{equation*}
|R_{x}^{\lambda,k}(f*p_{u})(\omega)|\le \sup_{0<y<x} |R_{y}^{\lambda,k^{\prime}}(f*p_{u})(\omega)|,
\end{equation*}
which proves the claim for all $k>0$.
\end{proof}

\begin{proof}[Proof of Theorem \ref{maxitransinternaltype1}]
Combine (\ref{besiconorm2}) from Corollary \ref{differentiation} with Theorem \ref{maxitranslatedtypeI}.
\end{proof}

Using the next lemma, Theorem \ref{maxitranslatedtypeII} follows from Theorem \ref{maxitranslatedtypeI}.

\begin{Lemm} \label{lemmasecond} Let $0<k\le 1$ and $u>0$. Then for every $\varepsilon>0$ there is a constant $C=C(k,u,\varepsilon)$ such that for all $x>0$ and complex sequences $(a_{n})$ we have
$$\Big|\sum_{\lambda_{n}<x} a_{n}e^{-(u+\varepsilon)\lambda_{n}}(1-e^{\lambda_{n}-x})^{k} \Big| \le C
\sup_{0<y\leq x} \Big| e^{-uy}\sum_{\lambda_{n}<y} a_{n}\Big(1-\frac{\lambda_{n}}{y}\Big)^{k} \Big|.$$
\end{Lemm}
We follow a similar strategy as in the previous proof of Lemma \ref{Abeltype1}, use the following identity from \cite[\S IV.2, p. 21]{HardyRiesz}
\begin{equation} \label{idsecond}
\sum_{\lambda_{n}<x} a_{n}\Big(1-e^{\lambda_{n}-x}\Big)^{k}=ke^{-xk}  \int_{1}^{e^{x}} \Big( \sum_{e^{\lambda_{n}}<t} a_{n}\Big)(e^{x}-t)^{k-1} dt\,,
\end{equation}
and also some ideas from \cite[Proof of Theorem 20, p. 33]{HardyRiesz}.
\begin{proof}[Proof of Lemma \ref{lemmasecond}]
  Let $u,\varepsilon>0$. By Lemma \ref{Abeltype1} it suffices to prove
 \begin{equation} \label{sunshine}
 \Big|\sum_{\lambda_{n}<x} a_{n}e^{-(u+\varepsilon)\lambda_{n}}(1-e^{\lambda_{n}-x})^{k} \Big| \le
 C \sup_{0<y\leq x} \Big|\sum_{\lambda_{n}<y} a_{n}e^{-u\lambda_{n}} \Big(1-\frac{\lambda_{n}}{y}\Big)^{k} \Big|.
 \end{equation}
 Moreover, let us for simplicity write
 \begin{equation*}
 R_{y}^{\lambda,k}(u)=\sum_{\lambda_{n}<y} a_{n}e^{-u\lambda_{n}} \left(1-\frac{\lambda_{n}}{y}\right)^{k}.
 \end{equation*}
 We use the following identity (see again the beginning of the proof of Lemma \ref{Abeltype1})
\begin{align*}
\sum_{\lambda_{n}<x} a_{n}e^{-(u+\varepsilon)\lambda_{n}}(e^{x}-e^{\lambda_{n}})^{k}
&
=-\int_{1}^{e^{x}}\Big(\sum_{\lambda_{n}<\log(l)} a_{n}e^{-u\lambda_{n}} \Big) \frac{d}{dl} \big(l^{-\varepsilon} (e^{x}-l)^{k} \big) dl  \\ & = ke^{-\varepsilon x} \int_{1}^{e^{x}} \Big(\sum_{\lambda_{n}<\log(l)} a_{n}e^{-u\lambda_{n}} \Big) (e^{x}-l)^{k-1} dl\\ &-\int_{1}^{e^{x}} \Big(\sum_{\lambda_{n}<\log(l)} a_{n}e^{-u\lambda_{n}} \Big)  h'(l) dl=: A+B\,,
\end{align*}
where $h(l):=(l^{-\varepsilon}-e^{-\varepsilon
x})(e^{x}-l)^{k}$.
Let us first deal with the summand A. By substitution with $t=\log(l)$ we obtain
\begin{align*}
A=ke^{-\varepsilon x} \int_{0}^{x} \Big(\sum_{\lambda_{n}<t} a_{n}e^{-u\lambda_{n}} \Big)  (e^{x}-e^{t})^{k-1} e^{t} dt.
\end{align*}
Since the positive function $G(t):=\left(\frac{e^{x}-e^{t}}{x-t}\right)^{k-1}e^{t}$, where $0<t<x$, is increasing with  $\lim_{r\to x} G(r)=e^{xk}$, by the second mean value theorem (applied separately to the real and imaginary part) there are $0<\xi_{1},\xi_{2}<x$ such that
\begin{align*}
A=
&
ke^{xk}e^{-\varepsilon x}\int_{\xi_{1}}^{x} \Big(\sum_{\lambda_{n}<t} Re(a_{n})e^{-u\lambda_{n}} \Big) (x-t)^{k-1}dt\\ &+ ike^{xk}e^{-\varepsilon x}\int_{\xi_{2}}^{x}\Big(\sum_{\lambda_{n}<t} Im(a_{n})e^{-u\lambda_{n}} \Big) (x-t)^{k-1}dt ,
\end{align*}
and so by \cite[Lemma 7, p. 28]{HardyRiesz}
\begin{align*}
|A|&\le 2ke^{-\varepsilon x}e^{xk} \max_{j=1,2}\left|\int_{\xi_{j}}^{x} \Big(\sum_{\lambda_{n}<t} a_{n}e^{-u\lambda_{n}} \Big) (x-t)^{k-1}dt \right|  \\ &\le C e^{xk} e^{-\varepsilon x} \sup_{0<y\leq x} |y^{k}R_{y}^{\lambda,k}(u)| \le C_{1} e^{xk} \sup_{0<y\leq x} |R_{y}^{\lambda,k}(u)|.
\end{align*}
Now we consider the second summand $B$, and define  $g(t):=h^{\prime}(e^{t})e^{t}$, where $0<t<x$. Then  the substitution  $t=\log(l)$ and partial integration (use again \eqref{capri}) give
\begin{align*}
B&=-\int_{1}^{e^{x}} \Big(\sum_{\lambda_{n}<\log(l)} a_{n}e^{-u\lambda_{n}} \Big) h^{\prime}(l) dl
\\ &
=-\int_{0}^{x} \Big(\sum_{\lambda_{n}<t} a_{n}e^{-u\lambda_{n}} \Big)h^{\prime}(e^{t})e^{t} dt=\int_{0}^{x}\Big(\sum_{\lambda_{n}<t} a_{n}e^{-u\lambda_{n}}(t-\lambda_{n})\Big)g^{\prime}(t) dt.
\end{align*}
Let now $0<k<1$. Then using  \eqref{motor} and Fubini's theorem  we finally end up with
\begin{align*}
B= C(k)\int_{0}^{x} y^{k} R_{y}^{\lambda,k}(u)\int_{y}^{x} g^{\prime}(t)(t-y)^{-k} dt dy,
\end{align*}
where
\begin{align*}
g^{\prime}(t)=
&
(e^{x}-e^{t})^{k} \varepsilon^{2} e^{-\varepsilon t} \\ &+ (e^{x}-e^{t})^{k-1} \varepsilon \left( (k-1)e^{-\varepsilon t}+ e^{-\varepsilon t}e^{t} -ke^{t}(e^{-\varepsilon t}-e^{-\varepsilon x})+k e^{-\varepsilon t} e^{t} \right) \\ &+ (e^{x}-e^{t})^{k-2} k(k-1)e^{t} (e^{-\varepsilon t}-e^{-\varepsilon x}).
\end{align*}
Estimating straight forward there is a constant $C_{2}= C_2(k, \varepsilon)$ such that
\begin{align*}
C_{2}^{-1}|g^{\prime}(t)|&\le (x-t)^{k} e^{xk}e^{-\varepsilon t} \\ &~~+ (x-t)^{k-1} e^{x(k-1)} \big( e^{-\varepsilon t} +e^{-\varepsilon t}e^{t}+e^{t}e^{-\varepsilon t} (x-t)+e^{- \varepsilon t} e^{t} \big)  \\&~~+ (x-t)^{k-2} e^{x(k-2)} e^{t} e^{-\varepsilon t} (x-t) \\ &\le 2(x-t)^{k}e^{xk} e^{-\varepsilon t}+ 4(x-t)^{k-1}e^{xk}e^{-\varepsilon t} \\ &= e^{xk}e^{-\varepsilon t}\big( 2(x-t)^{k}+4(x-t)^{k-1}\big ) .
\end{align*}
Hence, following the estimates from the end of the proof of Lemma \ref{Abeltype1} (compare the bound for $|g^{\prime}(t)|$ with the bound for $|h^{\prime \prime}(t)|$ from (\ref{bound}))  we conclude that
\begin{equation*}
e^{-xk}|B| \le C_{3} \sup_{0<y<x} |R_{y}^{\lambda,k}(u)|.
\end{equation*}
Finally (\ref{sunshine}) follows, since
\begin{equation*}
\Big|\sum_{\lambda_{n}<x} a_{n}e^{-(u+\varepsilon)\lambda_{n}}(1-e^{\lambda_{n}-x})^{k} \Big| \le e^{-xk} (|A|+|B|)\le  C_{4} \sup_{0<y<x} |R_{y}^{\lambda,k}(u)|.
\end{equation*}
Note that the case $k=1$ again follows the same lines without using (\ref{motor}).
\end{proof}

\begin{proof}[Proof of Theorem  \ref{maxitranslatedtypeII}]  Lemma \ref{lemmasecond} and (\ref{meet}) assure that for $0<k\le 1$ we have
\begin{equation} \label{regen}
\left(\int_{G} \sup_{x>0} |S_{x}^{\lambda,k}(f*p_{u})(\omega)|^{p} d\omega \right)^{\frac{1}{p}} \le C(u,k) \|f\|_{p}.
\end{equation}
 So let $k>1$ and write $k=l+k^{\prime}$, where $l\in \mathbb{N}$ and $0<k^{\prime}\le 1$. It suffices to show that \begin{align*}
|S^{\lambda,k}_{x}(f*p_{u})(\omega)|
\le  \sup_{0<y<x} |S_{y}^{\lambda,k^{\prime}}(f*p_{u})(\omega)|\,.
\end{align*}
Indeed, by definition and \eqref{gammaidentity} we have
\begin{equation*}
|S^{\lambda,k}_{x}(f)(\omega)|=|R_{e^x}^{e^\lambda,k}(f)(\omega)| \leq \sup_{0<t<e^x} |R_{t}^{e^\lambda,k'}(f)(\omega)|
= \sup_{0<y<x} |S_{y}^{\lambda,k^{\prime}}(f*p_{u})(\omega)|\,. \qedhere
\end{equation*}
\end{proof}

\subsection{Proof of the Theorems  \ref{fejer} and  \ref{inftycaseTheo}, and Corollary \ref{shiftcontinuous}} \label{fast}
The following observation is an important tool of both proofs.
\begin{Lemm} \label{measures}
 Let $\lambda$ be a frequency, $k>0$ and $(G,\beta)$ a $\lambda$-Dirichlet group.
Then there is a constant $C >0$ such that for all $x>0 $ there is a measure $\mu_x \in M(G)$ with  $\|\mu_x\|\le C$ and such that for all $n \in \N$
$$\widehat{\mu_x}(h_{\lambda_{n}})=\begin{cases} \Big(1-\frac{\lambda_{n}}{x}\Big)^{k}&, \text{if} ~~\lambda_{n}< x,\\
0&, \text{if} ~~\lambda_{n}\ge x. \end{cases}$$
\end{Lemm}

\begin{proof}[Proof of Lemma \ref{measures}]
The case $p=\infty$ of Theorem \ref{maximalineqtypeI} implies that there is a constant $C>0$ such that  for all $x > 0$ and all $f \in H_\infty^\lambda(G)$
$$\Big|\sum_{\lambda_{n}<x} \widehat{f}(h_{\lambda_{n}}) \Big(1-\frac{\lambda_{n}}{x}\Big)^{k}\Big|\le\|R_{x}^{\lambda,k}(f)\|_{\infty}\le C\|f\|_{\infty}.$$
Denote the subspace of all continuous functions in $H_\infty^\lambda(G)$ by   $C^{\lambda}(G)$, and fix some $x > 0$.  Then the bounded functional
$$T_x\colon C^{\lambda}(G) \to \C, ~~ f \mapsto \sum_{\lambda_{n}<x} \widehat{f}(h_{\lambda_{n}}) \Big(1-\frac{\lambda_{n}}{x}\Big)^{k}$$
has  norm $\leq C$, and satisfies   $T_x(h_{\lambda_{n}})=\big(1-\frac{\lambda_{n}}{x}\big)^{k}$ for  $\lambda_n < x$
and $T_x(h_{\lambda_{n}})=0$ for   $\lambda_n \ge  x$.
 By the  Hahn-Banach theorem there is $\widetilde{T_x} \in (C(G))^{\prime}$ extending  $T_x$ with equal norm, and
 then also the linear operator
 $$R_x\colon C^{\lambda}(G) \to \C, ~~
 R_x(f):=\overline{\widetilde{T_x}\Big(\overline{f}\Big)}
 $$
 has norm $\leq C$, and satisfies $R_x(\overline{h_{\lambda_n}}) =
 \big(1-\frac{\lambda_{n}}{x}\big)^{k}$ for  $\lambda_n < x$
and $R_x(\overline{h_{\lambda_n}})=0$ for   $\lambda_n \ge  x$.
Hence  the Riesz representation theorem assures the existence of a  measure
   $\mu_x$  with norm $\leq C$ which, since $\widehat{\mu_x}(h_{\lambda_{n}}) = R_x(\overline{h_{\lambda_n}})$
   for all $n$, has
   the desired Fourier coefficients.
\end{proof}

\begin{proof}[Proof of Theorem \ref{fejer}]Note that for any polynomial $P \in H_1^\lambda(G)$ we have
$$\lim_{x\to \infty} P*\mu_x=P \,\,\, \text{in $H_1^\lambda(G)$}\,,$$
 where $\mu_x$ is the measure from Lemma \ref{measures}. Now, given  $f \in H_{1}^{\lambda}(G)$ and
 $\varepsilon >0$, choose by density a polynomial $P$ such that $\|f-P\|_{1}\le \varepsilon$. Then for large $x$
 (and  the constant $C$ from Lemma \ref{measures})
\begin{align*}
&
\|f-f*\mu_x\|_{1}
\\ &
\le \|f-P\|_{1}+\|P-P*\mu_x\|_{1}+\|(P-f)*\mu_x\|_{1}\le \varepsilon\big(2+C\big)\,. \qedhere
\end{align*}
\end{proof}
Observe, that the counterexamples of Remark \ref{approxsecondmeansnorm} show that the variant of Lemma \ref{measures} for second Riesz means does not hold in the sense that there are no measures $\mu_{x}\in M(G)$, $x>0$, $\|\mu_{x}\|\le C$ for some $C>0$, such that
$$\widehat{\mu_x}(h_{\lambda_{n}})=\begin{cases} \Big(1-e^{\lambda_{n}-x}\Big)^{k}&, \text{if} ~~\lambda_{n}< x,\\
0&, \text{if} ~~\lambda_{n}\ge x. \end{cases}$$

In the proof of Theorem \ref{inftycaseTheo} we take advantage  of Lemma \ref{measures} and combine it with an  estimate for the abscissa of uniform summability by  Riesz means. Given $D=\sum a_{n}e^{-\lambda_{n}s}$ and $k\ge 0$, we define $\sigma_{u}^{\lambda,k}(D)$ to be the infimum of all $\sigma \in \R$ such that $D$ is uniformly $(\lambda, k)$-summable on $[Re>\sigma]$, i.e. the limit
$$\lim_{x\to \infty} \sum_{\lambda_{n}<x} a_{n} \Big(1-\frac{\lambda_{n}}{x}\Big)^{k} e^{-\lambda_{n}s}$$
exists uniformly on $[Re>\sigma]$.
We are going to make use of the following Bohr-Cahen type formula proved in \cite{Schoolmann1}:
\begin{equation} \label{BohrCahen}
\sigma_{u}^{\lambda,k}(D)\le \limsup_{x\to \infty} \frac{\log\Big( \|R_{x}^{\lambda,k}(D)\|_{\infty} \Big)}{x}\,,
\end{equation}
with equality whenever $\sigma_{u}^{\lambda,k}(D) \ge 0$.

\begin{proof}[Proof of Theorem \ref{inftycaseTheo}] Let us start defining a contractive coefficient preserving
mapping
$$\Psi \colon \mathcal{H}_{\infty}^{\lambda}[Re>0] \to H_{\infty}^{\lambda}(G)\,.$$
Take $F \in \mathcal{H}_{\infty}^{\lambda}[Re>0]$. Then (as described in Section \ref{aap}), given  $\sigma>0$, the uniformly almost periodic function $F_{\sigma} = F(\sigma + i \cdot)$  is a uniform limits of polynomials of the form $P^{N}_{\sigma}(t)=\sum_n b_{n}^Ne^{-\lambda_{n}it}$. Hence by density of $\beta: \R \to G$, the polynomials
$p_{\sigma}^{N}:=\sum_n b_{n}^Nh_{\lambda_{n}}$ form a Cauchy sequence in $H_{\infty}^{\lambda}(G)$ with limit, say, $f_{\sigma}$ with $\|f_{\sigma}\|_{\infty}=\|F_{\sigma}\|_{\infty}\le \|F\|_{\infty}$. Then by a standard   weak compactness argument there is $\Psi(h): =f \in H_{\infty}^{\lambda}(G) \subset L_\infty(G)$
with $\|f\|_{\infty}\le \|F\|_{\infty}$, which is the weak star limit of some subsequence of $(f_{\frac{1}{n}})_{n}$
(use that the unit ball of $L_\infty(G)$, being the dual  of  $L_1(G)$, endowed with its weak star topology is
compact and  metrizable). Then a simple argument shows that  $a_{\lambda_n}(F)=\widehat{f}(h_{\lambda_n})$ for all $n$, i.e.
$\Psi$ is indeed an  coefficient preserving  contraction.\\
In order to show that $\Psi$ is  in fact  an isometry onto, take  $f \in H_{\infty}^{\lambda}(G)$. Using the measures $\mu_x$ from Lemma \ref{measures}  and the fact that $\beta$ has dense range,  we for all $x>0$
have
\begin{align} \label{bernal}
\begin{split}
&\sup_{t \in \R} \bigg|\sum_{\lambda_{n}<{x}} \widehat{f}(h_{\lambda_{n}}) \Big(1-\frac{\lambda_{n}}{x}\Big)e^{-\lambda_{n}it}  \bigg|
\\ &=\left\| \sum_{\lambda_{n}<{x}} \widehat{f}(h_{\lambda_{n}}) \Big(1-\frac{\lambda_{n}}{x}\Big) h_{\lambda_{n}} \right\|_{\infty}=\|f*\mu_x\|_{\infty} \le C \|f\|_{\infty}.
\end{split}
\end{align}
Hence \eqref{BohrCahen} applied to $D:=\sum \widehat{f}(h_{\lambda_{n}})  e^{-\lambda_{n}s}$ shows that
$\sigma_u^{\lambda,1}(D)\leq 0$, and this  in particular proves   that
$$F(s):= \lim_{x\to \infty} \sum_{\lambda_{n}<x} \widehat{f}(h_{\lambda_{n}}) \Big(1-\frac{\lambda_{n}}{x}\Big)e^{-\lambda_{n}s} \colon [Re>0] \to \C$$
defines a holomorphic function on $[Re>0]$ which converges uniformly on every smaller half-space $[Re>\sigma], \sigma>0$.
As explained in Section  \ref{aap} we may deduce  that all functions
$F_\sigma = F (\sigma + i \cdot), \,\sigma >0$ are uniformly almost periodic
with Bohr coefficients $a_{\lambda_n}(F)=\widehat{f}(h_{\lambda_{n}})$ for all $n$ and zero else.
 It remains to show, that $F$ is bounded. By equation \eqref{coro1.3pos} from
Corollary \ref{Dirichletseries} there is some $\omega \in G$, such that for all $\sigma>0$ and  almost all $t\in \R$ we have
\begin{equation} \label{continuity}
f_{\omega}*P_{\sigma}(t)=\lim_{x\to \infty}  \sum_{\lambda_{n}<x} \widehat{f}(h_{\lambda_{n}}) h_{\lambda_n}(\omega) \Big(1-\frac{\lambda_{n}}{x}\Big)e^{-\lambda_{n}(\sigma+it)};
\end{equation}
note that here both sides form continuous functions, and hence the equality in fact holds  for every $t\in \R$.
On the other hand  we deduce from the rotation invariance of the Haar measure  that
\[
\left\| \sum_{\lambda_{n}<{x}} \widehat{f}(h_{\lambda_{n}}) \Big(1-\frac{\lambda_{n}}{x}\Big) h_{\lambda_{n}} \right\|_{\infty}
=
\left\| \sum_{\lambda_{n}<{x}} \widehat{f}(h_{\lambda_{n}}) h_{\lambda_n}(\omega) \Big(1-\frac{\lambda_{n}}{x}\Big) h_{\lambda_{n}} \right\|_{\infty}\,,
\]
and therefore another application of  \eqref{bernal} and \eqref{BohrCahen} shows that the   vertical limits
 $D^{\omega}=\sum \widehat{f}(h_{\lambda_{n}}) h_{\lambda_n}(\omega)e^{-\lambda_{n}s}$  are uniformly summable by first $(\lambda,1)$-Riesz means on all half-planes $[Re>\sigma]$ with $\sigma>0$. All together this  implies
\begin{align*}
\|F\|_{\infty}&=\sup_{\sigma>0} \|F_{\sigma}\|_{\infty}=\sup_{\sigma>0}\lim_{x \to \infty} \left\|\sum_{\lambda_{n}<x} \widehat{f}(h_{\lambda_{n}}) \Big(1-\frac{\lambda_{n}}{x}\Big)e^{-(\sigma+i\cdot)\lambda_{n}} \right\|_{\infty}\\ &=\sup_{\sigma>0}\lim_{x \to \infty} \left\|\sum_{\lambda_{n}<x} \widehat{f}(h_{\lambda_{n}}) h_{\lambda_n}(\omega)\Big(1-\frac{\lambda_{n}}{x}\Big)e^{-(\sigma+i\cdot)\lambda_{n}}\right\|_{\infty}\\ &\le\sup_{\sigma>0} \|f_{\omega}*P_{\sigma}\|_{\infty} \le \|f_{\omega}\|_{\infty}=\|f\|_{\infty}\,,
\end{align*}
 and so  $\Psi$ is indeed an isometry onto.
\end{proof}

\begin{proof}[Proof of Corollary \ref{shiftcontinuous}]
Let $f\in H_{\infty}^{\lambda}(G)$ and $F:=\Psi^{-1}(f)$, where $\Psi$ is the mapping from Theorem \ref{inftycaseTheo}. Then for every $u>0$ the restriction $F_{u}=F(u+ i\cdot)$ is uniformly almost periodic on $\R$. So there is $g_u \in C(G)\cap H_{\infty}^{\lambda}(G)$ such that $F_{u}=g_u\circ \beta$, and $\widehat{g_u}(h_{\lambda_{n}})=a_{\lambda_{n}}(F)e^{-u\lambda_{n}}=\widehat{f}(h_{\lambda_{n}})e^{-u\lambda_{n}}$ for all $n$. Hence $g_u=f*p_{u}$ (compare Fourier coefficients). The second statement follows by approximation with polynomials of the form $\sum_{n=1}^{N} b_{n}h_{\lambda_{n}}$, which are dense in $C(G)\cap H_{\infty}^{\lambda}(G)$ (see \cite[\S 8.7.3]{Rudin62}).
\end{proof}


\begin{thebibliography}{10}
\bibitem{Bayart}
F. Bayart: \emph{Hardy spaces of Dirichlet series and their compostion operators}, Monatsh. Math. {136} (2002) 203-236
\bibitem{Besicovitch}
A. S. Besicovitch: \emph{Almost periodic functions}, Dover publications (1954)
\bibitem{Defant} A. Defant, D. Garc\'ia, M. Maestre, and P. Sevilla Peris: \emph{Dirichlet series and holomorphic functions in high dimensions}, to appear  in: New Mathematical Monographs Series, Cambridge University Press (2019)
\bibitem{DefantSchoolmann1} A. Defant and I. Schoolmann: \emph{$\mathcal{H}_{p}$-theory of general Dirichlet series}, preprint 2019
\bibitem{DefantSchoolmann2} A. Defant and I. Schoolmann: \emph{On Helson's theorem for general Dirichlet series}, in preparation  2019


    \bibitem{Duy} T.K. Duy, \emph{On convergence of Fourier series of Besicovitch almost
periodic functions},
    Lithuanian Math. J. 53,3 (2013) 264-279

\bibitem{Evans} L.C. Evans and R.F. Gariepy, \emph{Measure Theory and Fine Properties of Functions}, Studies in Advanced Mathematics,(1992).


    \bibitem{Fefferman}
    C. Fefferman:
    \emph{On the convergence of multiple Fourier series}, Bull. Amer. Math. Soc.
    77 (1971) 744-745.
\bibitem{Grafakos1} L. Grafakos: \emph{Classical Fourier analysis}, Graduate Texts in Mathematics {249} (2014).
\bibitem{HardyRiesz}
G. H. Hardy and M. Riesz: \emph{The general theory of Dirichlet series}, Cambridge Tracts in Mathematics and Mathematical Physics {18} (1915).

\bibitem{HLS}
H. Hedenmalm, P. Lindqvist, and K. Seip: \emph{A Hilbert space of Dirichlet series and systems of dilated function in $L^{2}(0,1)$}, Duke Math. J. 86 (1) (1997) 1-37.
\bibitem{HedenmalmSaksman}
H. Hedenmalm and  E. Saksman: \emph{Carleson's convergence theorem for Dirichlet series},
Pacific J. of Math. 208 (2003) 85-109.

\bibitem{Helson3}
H. Helson: \emph{Compact groups and Dirichlet series}, Ark. Mat. {8} (1969) 139-143.
\bibitem{Helson}
H. Helson: \emph{Dirichlet series}, Regent Press (2005).
\bibitem{Landau}
E. Landau: \emph{\"{U}ber die gleichm\"{a}\ss ige Konvergenz Dirichletscher Reihen}, J. Reine Angew. Math. {143} (1921) 203-211.
\bibitem{Rudin} W. Rudin: \emph{Real and complex analysis},  McGraw-Hill (1987).
\bibitem{Rudin62}
W. Rudin: \emph{Fourier analysis on groups}, Interscience Publishers (1962).
\bibitem{Stein} E. M. Stein and R. Shakarchi, \emph{Real analysis}, Princeton lectures in analysis III, Princeton university Press (2005).
\bibitem{Schoolmann1}
 I. Schoolmann: \emph{On Bohr's theorem for general Dirichlet series}, to appear in Math. Nachr. 2019.
\bibitem{QQ}
H. Queff\'{e}lec and M. Queff\'{e}lec: \emph{Diophantine approximation and Dirichlet series}, Hindustan Book Agency, Lecture Note 2 (2013).



\end{thebibliography}
\end{document}